\documentclass[sn-mathphys]{sn-jnl}

\jyear{2023}%

\usepackage{amsmath,amsbsy,amssymb,amsthm}
\usepackage{multirow,multicol,tabularx,booktabs}
\usepackage{graphicx,placeins,color,url}
\usepackage{bbm,bm}

\usepackage[shortlabels]{enumitem}
\usepackage{mathrsfs}
\usepackage{textcase}

\catcode`\|=12\relax

\usepackage{tikz}
\usetikzlibrary{arrows,arrows.meta,shapes.arrows,positioning,calc}

\newtheorem{conj}{Conjecture}[section]
\newtheorem{thm}[conj]{\bf Theorem}
\newtheorem{defi}[conj]{\bf Definition}

\newtheorem{prop}[conj]{\bf Proposition}
\newtheorem{lemma}[conj]{\bf Lemma}

\def\to{\rightarrow}

\def\Cov{\operatorname{Cov}}
\def\Var{\operatorname{Var}}

\def\Rank{\operatorname{Rank}}

\def\Vect{\operatorname{Vect}}
\def\dim{\operatorname{dim}}

\def\Pois{\operatorname{Pois}}
\def\Exp{\operatorname{Exp}}
\def\Ber{\operatorname{Ber}}

\def\EssSup{\operatorname{ess \, sup}}

\def\Ac{\mbox{$\mathcal A$}}
\def\Bc{{\mathscr B}}
\def\Cc{{\mathscr C}}

\def\Mc{{\mathcal M}}
\def\Nc{\mbox{$\mathcal N$}}

\def\Xc{{\mathcal X}}

\def\Gb{{\mathbb G}}

\def\NN{{\mathbb N}}

\def\Rb{{\mathbb R}}

\def\e{ {\bf e}}
\def\t{ {\bf t}}

\def\v{ {\bf v}}
\def\x{ {\bf x}}

\def\X{ {\bf X}}

\def\Z{ {\bf Z}}

\def\1{\mathbbm{1}}

\def\support{{\rm Supp}}

\def\wt{\widetilde}

\newcommand{\Vphi}{{V_\phi}}
\newcommand{\Rphi}{{R_\phi}}
\newcommand{\Valpha}{{V_\alpha}}
\newcommand{\Ralpha}{{R_\alpha}}

\newcommand{\Ralphabig}[3]{{R_\alpha \big(\hspace{-0.1em}
    #1 \, \big| \, 
    #2 \, , \,
    #3 \big)}}

\newcommand{\RalphaBigg}[3]{{R_\alpha \Bigg(
    #1 \, \Bigg| \, 
    #2 \, , \,
    #3 \Bigg)}}

\begin{document}

\title[Codivergences and information matrices]{Codivergences and information matrices}


\author*[1]{\fnm{Alexis} \sur{Derumigny}}\email{a.f.f.derumigny@tudelft.nl}


\author[2]{\fnm{Johannes} \sur{Schmidt-Hieber}}\email{a.j.schmidt-hieber@utwente.nl}

\affil[1]{\orgdiv{Department of Applied Mathematics}, \orgname{Delft University of Technology}, \orgaddress{\street{Mekelweg 4}, \postcode{2628 CD}, \city{Delft}, \country{The Netherlands}}}


\affil[2]{\orgname{University of Twente}, \orgaddress{\street{Drienerlolaan 5}, \city{Enschede}, \postcode{7522 NB}, \country{The Netherlands}}}


\abstract{We propose a new concept of codivergence, which quantifies the similarity between two probability measures $P_1, P_2$ relative to a reference probability measure $P_0$. In the neighborhood of the reference measure $P_0$, a codivergence behaves like an inner product between the measures $P_1-P_0$ and $P_2-P_0$. Codivergences of covariance-type and correlation-type are introduced and studied with a focus on two specific correlation-type codivergences, the $\chi^2$-codivergence and the Hellinger codivergence.} We derive explicit expressions for several common parametric families of probability distributions. For a codivergence, we introduce moreover the divergence matrix as an analogue of the Gram matrix. It is shown that the $\chi^2$-divergence matrix satisfies a data-processing inequality.



\maketitle

\newpage 
\section{Introduction}

One of the objectives of information geometry is to measure distances or angles in statistical spaces, usually for parametric models.
This is often done by the use of a divergence, generating a Riemannian manifold structure on the considered space of distributions, see \cite{amari2016information,ay2017information,nielsen2020elementary}.
Divergences between probability measures quantify a certain notion of difference between them. Divergences are in general not symmetric, as opposed to distances. Famous examples of divergences include the Kullback-Leibler divergence, the $\chi^2$-divergence, and the Hellinger distance.

\medskip

In this article, we are interested in defining a local notion of inner product between two probability measures in the neighborhood of a given reference probability measure $P_0$. This allows us to identify different directions relative to $P_0$, and to give some meaning to the ``angle'' between these directions. Contrary to most of the previous work on finite-dimensional Riemannian manifolds spanned by specific parametric statistical models, we do not require any parametric restrictions on the considered probability measures.

\medskip

Motivation and application of our approach is the recently considered generic framework to derive lower bounds for the trade-off between bias and variance in nonparametric statistical models \cite{2022lowerBoundsBV}. The key ingredient in this lower bound strategy are so-called change of expectation inequalities that relate the change of the expected value of a random variable with respect to different distributions to the variance and also involve the divergence matrices examined in this work. Another possible area of application are the lower bounds for statistical query algorithms, see e.g. \cite{feldman2017statistical,diakonikolas2017statistical}.

\medskip

Regarding work on infinite-dimensional information geometry, \cite{cena2007exponential,pistone1995infinite} studied the manifold generated by all probability densities connected to a given probability density.
\cite{pistone2013nonparametric} reviews a more general theory on infinite-dimensional statistical manifolds given a reference density.
Another line of work \cite{holbrook2017nonparametric, newton2012infinite, newton2016infinite,newton2019class} seeks to define infinite-dimensional manifolds, with applications to Bayesian estimation and the choice of priors. \cite{srivastava2007riemannian} consider different possible structures on the set of probability densities on $[0,1]$.

\medskip

The article is structured as follows.
In Section~\ref{sec:codiv}, we introduce a general concept of codivergence, study specific properties of codivergences on the space of probability measures, and discuss specific (classes of) codivergences. Section~\ref{sec:div_matrices} considers the construction of divergence matrices from a given codivergence. Section~\ref{sec:data_processing} is devoted to the data-processing inequality that holds for the $\chi^2$-divergence matrix introduced in Section~\ref{sec:div_matrices}, thereby generalizing the usual data-processing inequality for the $\chi^2$-divergence.
Section~\ref{sec.explicit} provides derivations of explicit expressions for a class of codivergences applied to common parametric models. Elementary facts on ranks from linear algebra are collected in Section~\ref{sec:useful_lemmas}.

\medskip

\noindent
\textit{Notation:} 
If $P$ is a probability measure and $\X$ a random vector, we write $E_P[\X]$ and $\Cov_P(\X)$ for the expectation vector and covariance matrix with respect to $P,$ respectively.

\section{Codivergences}
\label{sec:codiv}

\subsection{Abstract framework and definition}

We start by recalling the definition of a divergence \cite[Definition 1.1]{amari2016information}. This definition is situated within the framework of a $d$-dimensional differentiable manifold $\Xc$ with an atlas $(U_i, \varphi_i)$.
Formally, this means that the $(U_i)$ are an open cover of the topological space $\Xc$, and $\varphi_i: U_i \to \Rb^d$ are isomorphisms such that
$\varphi_j \circ \varphi_i^{-1}$ is $C^1$ on $\varphi_i(U_i \cap U_j) \to \varphi_j(U_i \cap U_j)$ \cite{lang2012differential}.

\begin{defi}
    A divergence $D$ on a $d$-dimensional differentiable manifold $\Xc$ is a function $\Xc^2 \to \Rb_+$ satisfying
    \vspace{0.2em}
    \begin{enumerate}
        \item[(i)] $\forall P, Q \in \Xc,$ $D(P|Q) = 0$ if and only if $P = Q$.
        \vspace{0.2em}

        \item[(ii)] For all $P \in \Xc$, for any chart $(U, \varphi)$ with $P \in U$, there exists a matrix $G = G(P)$ such that for any $Q \in U$,
        \begin{align}
            \hspace{-2em}
            D(P | Q) = \frac{1}{2}
            \big( \varphi(Q) - \varphi(P) \big)^T G 
            \big( \varphi(Q) - \varphi(P) \big) 
            + O\Big( \| \varphi(Q) - \varphi(P) \|^3 \Big).
            \label{eq:def:divergence}
        \end{align}
    \end{enumerate}
    \label{def:divergence}
\end{defi}

The matrix $G=G(P)$ may depend on the choice of coordinates $\varphi$. For the most common divergences, $G$ is symmetric, positive-definite and thus defines a scalar product on the tangent space at $P$.
Whereas a divergence measures the similarity between two elements $P,Q\in \Xc,$ we want to define codivergences measuring the angle $\sphericalangle P_1P_0P_2$ of $P_1,P_2\in \Xc$ relative to $P_0\in \Xc.$

\medskip

Equation \eqref{eq:def:divergence} states that the divergence $D(P|Q)$ is a quadratic form in terms of the local coordinates $\varphi(Q) \in \Rb^d,$ whenever $P$ and $Q$ are close. Generalizing to the infinite-dimensional case  requires to work with bilinear forms instead. Moreover for the infinite-dimensional setting, imposing an expansion of the form \eqref{eq:def:divergence} in every possible direction around $P \in \Xc$ is restrictive.
We therefore allow the quadratic expansion to hold in a possibly smaller bilinear expansion domain. Furthermore, we allow codivergences to attain the value $+ \infty.$ This is inspired by existing statistical divergences (such as $\chi^2$- or Kullback-Leibler divergences) that can also take the value $+\infty$. 
Therefore, imposing an expansion of the form \eqref{eq:def:divergence} globally may not be possible as the codivergence on the left-hand side of \eqref{eq:def:divergence} may take the value $+\infty$ in some directions away from $P$, while the right-hand side of \eqref{eq:def:divergence} is always finite.

\medskip

We now provide the definition of a codivergence if $\Xc$ is a subset of a real vector space.

\begin{defi} \label{def:codiv}
    Let $\Xc$ and $(E_{u})_{u\in \Xc}$ be a subset and a family of subspaces of a real vector space $E$, respectively.
    A function $(u,v,w) \in \Xc^3 \mapsto
    D(u | v, w) \in \Rb \cup \{ + \infty\}$ defines a \emph{codivergence} on $\Xc$ with \emph{bilinear expansion domain} $E_{u}$ at $u,$ if for any $u, v, w \in \Xc,$
    \vspace{0.2em}
    \begin{enumerate}
        \item[(i)] 
        $D(u | v, w) = D(u | w, v)$;
        \vspace{0.2em}
        
        \item[(ii)] $D(u | v, v) \geq 0$, with equality if $u = v$;
        \vspace{0.2em}
        
        \item[(iii)] there exists a bilinear map $\langle \cdot, \cdot \rangle_{u}$ defined on $E_{u}$, such that, for any $h, g \in E_{u}$ and for any scalars $s, t$ in some sufficiently small open neighborhood of $(0,0)$ (that may depend on $h$ and $g$) with respect to the Euclidean topology in $\Rb^2$, we have $(u + t h, \ u + sg) \in \Xc^2,$ $D \big(u  \, \big| \,
        u + t h \, , \, u + s g\big) < + \infty,$ and $D \big(u  \, \big| \,
        u + t h \, , \, u + s g\big)
        = ts \langle h, g \rangle_{u}
        + o(t^2 + s^2)$ as $(s,t) \to (0,0)$.
    \end{enumerate}
\end{defi}

The last part of the definition imposes that, locally around each $u$, the codivergence $(v, w) \mapsto D(u | v, w)$ is finite and behaves like a bilinear form
in the centered variables $(v-u, w-u).$
As a consequence, for a given $u$, the mapping $(v, w) \mapsto D(u | v, w)$ is Gateaux-differentiable on $\Xc^2$ at $(u, u)$ with Gateaux derivative $0$ in every direction $(h,g) \in E_u^2$.
Condition (iii) can moreover be understood as a second-order Taylor expansion at $(u, u)$ in the direction $(h,g)$. The mapping $(v, w) \mapsto D(u | v, w)$ needs, however, not to be twice Gateaux-differentiable at $(u, u)$ for $(iii)$ to hold.
This is analogous to usual counter-examples in analysis where a function may admit a second-order Taylor expansion at a given point without being twice differentiable.
Nevertheless, if $D \big(u  \, \big| \, u + t h \, , \, u + s g\big)$ is twice differentiable in $(t,s)$ at $(0,0)$, then the partial derivative
$\partial^2 D \big(u  \, \big| \, u + t h \, , \, u + s g\big)/ \partial t \partial s$ at $(0,0)$ must be equal to $\langle h, g \rangle_{u}$.
We refer to \cite{averbukh1967theory} for a discussion on higher-order functional derivatives. 

\medskip

We also provide a definition of codivergence if $\Xc$ is a differentiable Banach manifold, see \cite{lang2012differential,bourles2019fundamentals} for an introduction to Banach manifolds. Let $B$ be a Banach space and $\Xc$ be a Banach manifold modeled on $B$. This guarantees existence of a $B$-atlas $(U_i, \varphi_i)$ with $U_i$ an open cover of $\Xc$ and $\varphi_i: U_i \to B$ such that $\varphi_j \circ \varphi_i^{-1}$ is $C^1$ (with respect to the norm on $B$).

\medskip

This generalization can be useful in the case where the space $\Xc$ is not flat. Indeed, part (iii) of Definition~\ref{def:codiv} imposes that for $h \in E_u$, we must have $u + t h \in \Xc$ for $t$ small enough.
On the contrary, in the following definition we consider a more subtle case where the point $u$ may be approached on a smooth curve (not necessarily affine), under the assumption that $\Xc$ is a $B$-manifold.

\medskip

We first recall the construction of the tangent space via curves following \cite[Definition 2.21]{bourles2019fundamentals} and \cite[Section 2.1.1]{lee2022manifolds}:
for a fixed $u \in \Xc$, let $i$ be such that $u \in U_i$ and let $\Cc_u$ be the set of smooth curves $c$ such that $c: [-1,1] \to U_i$ and $c(0) = u$.
We define an equivalence relation $\sim$ on $\Cc_u$ by $c_1 \sim c_2$ if for all smooth real-valued functions $f$ on $U_i$,
we have $(f \circ c_1)'(0) = (f \circ c_2)'(0)$.
We define the tangent space at $u$ as the quotient set $T_u := \Cc_u /{\sim}$, which can be given a vector space structure isomorphic to $B$.

\medskip

We give a short outline of the main ideas to obtain this property.
Let $D$ denote the Fréchet differential operator.
For $\overline{c} \in T_u$ and $c$ a representative of the equivalence class $\overline{c}$, note that $\varphi_i \circ c: [-1,1] \to B$ is differentiable (by assumption on $c$); the mapping $D(\varphi_i \circ c)(0)$ is linear from $\Rb$ to $B$ and can therefore be identified with an element of $B$ itself; this element $D(\varphi_i \circ c)(0)$ also does not depend on the representative $c$.
This defines a mapping $\theta_u: T_u \mapsto B$ by $\theta_u(\overline{c}) := D(\varphi_i \circ c)(0)$.
It can be shown that $\theta_u$ is bijective.
Through its inverse $\theta_u^{-1}$ one can transport the vector space structure of $B$ on $T_u$, making it a real vector space too.

\begin{defi} \label{def:codiv_manifold}
    Let $\Xc$ be a $B$-manifold.
    A function $(u,v,w) \in \Xc^3 \mapsto
    D(u | v, w) \in \Rb \cup \{ + \infty\}$ defines a \emph{codivergence} on $\Xc$ with \emph{bilinear expansion domain} $E_{u}$ at $u,$ if for any $u, v, w \in \Xc,$
    \begin{enumerate}
        \item[(i)] $D(u | v, w) = D(u | w, v)$;
        \vspace{0.2em}
        
        \item[(ii)] $D(u | v, v) \geq 0$, with equality if $u = v$;
        \vspace{0.2em}
        
        \item[(iii)] $E_u$ is a subspace of the tangent space $T_u$ of $\Xc$ at $u$;
        \vspace{0.2em}
        
        \item[(iv)] there exists a bilinear map $\langle \cdot, \cdot \rangle_{u}$ defined on $E_{u}$.
        For any $\overline{g}, \overline{h} \in E_{u}$,
        for any representatives $g$ and $h$ of the respective equivalence classes $\overline{g}$ and $\overline{h},$
        and for any scalars $s, t$ in some sufficiently small open neighborhood of $(0,0)$ with respect to the Euclidean topology in $\Rb^2$
        (the neighborhood may depend on the choice of the representatives $g$ and $h$),
        we have
        $D \big(u  \, \big| \,
        h(t) \, , \, g(s) \big) < + \infty,$ and
        $D \big(u  \, \big| \,
        h(t) \, , \, g(s) \big)
        = ts \langle
        \overline{h}, \overline{g} \rangle_{u}
        + o(t^2 + s^2)$ as $(s,t) \to (0,0).$
    \end{enumerate}
\end{defi}

From a codivergence $D(u|v,w)$ that takes finite values on a finite-dimensional manifold and with bilinear expansions domains the tangent spaces, we can always construct a divergence by setting $v = w$. Then $D(u|v, v)$ behaves like a quadratic form in $v$ whenever $v$ is close to $u$. 

\medskip

If $\Xc$ is a $B$-manifold and a closed subspace of a vector space $E$, then the notions of codivergences in Definition~\ref{def:codiv} and Definition~\ref{def:codiv_manifold} coincide. This is because differentiable curves are, in first order, linear functions in a small enough neighborhood of $0.$

\medskip

For both definitions, a given space $\Xc$ and a given family of bilinear expansion domains $(E_{u})_{u \in \Xc}$, the set of codivergences on $\Xc$ is a convex cone.

\medskip

For an example covered by Definition~\ref{def:codiv_manifold} but not by Definition~\ref{def:codiv} assume that $\Xc$ is the unit circle. No codivergence can exist in the sense of Definition~\ref{def:codiv} with non-trivial bilinear expansion domains $(E_u)$. An example of a codivergence on $\Xc = \{ e^{i\theta}, \theta \in \Rb\}$ in the sense of Definition~\ref{def:codiv_manifold} is
\begin{align*}
    D(u | v , w)
    = \begin{cases}
        e^{(\theta_1 - \theta_0) (\theta_2 - \theta_0)} - 1,
        & \textnormal{if } v, w \in U_u, \\
        + \infty, & \textnormal{else,}
    \end{cases}
\end{align*}
where $u, v, w \in \Xc^3,$
$U_u := \{u e^{i\theta}, \, \theta \in (-\pi/2, \pi/2)\},$ 
$u = e^{i\theta_0},$
$v = e^{i\theta_1}$ and
$w = e^{i\theta_2}$
for some $\theta_0 \in \Rb$,
$\theta_1, \theta_2 \in [\theta_0 - \pi, \theta_0 + \pi)$.
Such a representation of $v$ and $w$ always exists and is unique since $[\theta_0 - \pi, \theta_0 + \pi)$ is a half-open interval of length $2 \pi$.
In this case, the tangent space $T_u$ of the circle at any point $u = e^{i \theta_0}$ is diffeomorphic to $\Rb$ and we will use this identification (denoted by the symbol ``$\simeq$''). Let $g,h \in T_u \simeq \Rb$ and assume $s,t \in \Rb$.
Then $g(s) = u e^{igs} \in U_u$ for $s$ small enough.
Similarly, $h(t) = u e^{iht} \in U_u$ for $t$ small enough.
So, $D \big(u  \, \big| \, h(t) \, , \, g(s) \big)$
is finite for all $(s,t)$ in a small enough neighborhood of $(0,0)$, and, whenever this is the case, we have
\begin{align*}
    D \big(u  \, \big| \, h(t) \, , \, g(s) \big)
    = D \big(u  \, \big| \, u e^{iht} \, , \,
    u e^{igs} \big)
    = e^{ht gs} - 1
    = ts \langle h, g \rangle_{u} + o(t^2 + s^2),
\end{align*}
where $\langle h, g \rangle_{u} = gh$ is the local bilinear form
(which in this example is independent of $u$)
and the bilinear expansion domain can be taken to be $E_u = T_u \simeq \Rb$.

\begin{figure}[tb]
    \centering
    \resizebox{!}{4cm}{


    \begin{tikzpicture}
        \shade[left color=gray!10,right color=gray!80] 
        (0, 0) to[out=-10,in=150] (4  , -3) --
        (12,1) to[out=150,in=-10] (5.5,1.7) -- cycle;

        \node at ($(3.5, 0)$) (P_0) {};
        \draw[fill] (P_0) circle (2pt);
        \node[above = of P_0, yshift = -1cm] {$P_0$};
        
        \node at ($(6.7, 0.5)$) (P_1) {};
        \draw[fill] (P_1) circle (2pt);
        \node[above = of P_1, yshift = -1cm] {$P_1$};
        
        \node at ($(6, -1.7)$) (P_2) {};
        \draw[fill] (P_2) circle (2pt);
        \node[above = of P_2, yshift = -1cm] {$P_2$};

        \draw (P_0.center) to [out=30,in=175] (P_1.center);
        \draw (P_0.center) to [out=-10,in=160] (P_2.center);

        \draw [red,thick,domain=-17:26]
        plot ({3.5 + 0.5 * cos(\x)}, {0.5 * sin(\x)}) ;
        
        \node[right = of P_0, xshift = -0.5cm]
        {$D(P_0 |P_1, P_2)$};
    \end{tikzpicture}
    }
    \caption{The codivergence between $P_1$ and $P_2$ at $P_0$ measures the position of $P_1$ and $P_2$ relative to $P_0.$}
    \label{fig:codiv}
\end{figure}
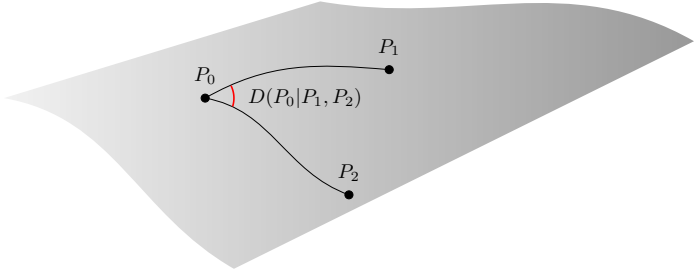

\subsection{Codivergences on the space of probability measures}

For the application to statistics, $E$ is the space of all finite signed measures on a measurable space $(\Ac, \Bc)$, and $\Xc$ is the space of all probability measures on $(\Ac, \Bc)$.
Probability measures form a convex subset of all signed measures $E$. Since $E$ is a vector space, the natural definition of a codivergence on $\Xc$ is Definition~\ref{def:codiv}.
A visual representation of such a codivergence is provided in Figure~\ref{fig:codiv}.

\medskip

In a next step, we characterize the bilinear expansion domains of a codivergence for $\Xc$ the space of probability measures. Given a probability measure $P_0 \in \Xc,$ we say that a function $h: \Xc \to \Rb$ is $P_0$-essentially bounded by a constant $C > 0$ if $P_0(\{x \in \Ac: |h(x)| \leq C \}) = 1$ and define $\EssSup_{P_0}|h| := \inf\{C > 0: \,
|h| \text{ is } P_0\text{-essentially bounded by } C \}.$ We will show that
\begin{align*}
    \Mc_{P_0} := \left\{ \mu \in E: \, 
    \mu \ll P_0, \,
    \int d\mu = 0, \,
    \EssSup_{P_0}\Big|\frac{d\mu}{dP_0}\Big| < + \infty
    \right\}
\end{align*}
is the largest bilinear expansion domain that any codivergence on $\Xc$ can have at $P_0$. The rationale is that $P_0+t\mu$ is otherwise not a probability measure. Indeed if $\mu \in \Mc_{P_0}$ has a density $h$ with respect to $P_0$, then the $P_0$-density $1 + t h$ is non-negative for given $t>0$ if and only if $h$ is larger than $-1/t$. Conversely, the density $1 - t h$ is non-negative for given $t <0$ if and only if $h$ is smaller than $1/t$.
This gives a link between a bound on $h = d\mu/dP_0$ and the non-negativity of the probability measure $P_0 + t \mu.$

\medskip

For every measurable set $A,$
\begin{align}
    (P_0 + t \mu)(A) = \int_{x\in A} d(P_0 + t \mu)(x)
    = \int_{x\in A} \big(1 + th(x)\big) dP_0(x).
\end{align}
The value of an integral is unchanged if the function to be integrated is modified on a $P_0$-null set. Therefore we only need the function $1 + th$ to be positive $P_0$-almost everywhere for $P_0 + t \mu$ to be a positive measure.

\begin{prop}
    For any codivergence $D$ on the space of probability measures $\Xc$, the bilinear expansion domain of $D$ at any probability measure $P_0 \in \Xc$ must be included in $\Mc_{P_0}$.
    Furthermore, every $\mu \in \Mc_{P_0}$ has a density $d\mu / dP_0$ with respect to $P_0$ such that
    $\EssSup_{P_0}|d\mu / dP_0|
    = 1 / a_*$
    with $a_* := \sup \{ a > 0:\, P_0 + t \mu \in \Xc \, \text{for all } t \in [-a , a]\} \in (0, +\infty]$ and the convention $1/+ \infty = 0$.
\label{prop:biggest_tangent_space}
\end{prop}


\begin{proof}[Proof of Proposition~\ref{prop:biggest_tangent_space}]
    We begin by proving the first part.
    Let $\mu$ be a finite signed measure belonging to the bilinear expansion domain at $P_0$ of some codivergence $D$ on the space of probability measures $\Xc$.
    For $a > 0$, we write $\mu \in R(a)$ if and only if $P_0 + t \mu \in \Xc$ for all $-a\leq t\leq a.$
    Since $\mu$ belongs to the bilinear expansion domain of $D$ at $P_0$, Definition~\ref{def:codiv}(iii) implies existence of an open neighborhood $T$ of $0$ such that for any $t \in T$, $P_0 + t \mu \in \Xc$. Therefore, there exists $a > 0$ with $\mu\in R(a)$.
    
    \medskip

    We now show that $\mu \in R(a),$ for some $a>0,$ implies $\mu \ll P_0$. The proof relies on the Jordan decomposition theorem for finite signed measures (e.g.\ Corollary 4.1.6 in \cite{MR3098996}). It
    states that every finite signed measure $\mu$ on a measurable space $(\Ac, \Bc)$ can be decomposed as
    \begin{align}
        \mu = \alpha_+ \mu_+ - \alpha_- \mu_-, \quad \text{with} \ \alpha_+, \alpha_- \geq 0,
        \label{Jordan_decomposition_theorem}
    \end{align}
    and $\mu_-, \mu_+$ orthogonal probability measures on $(\Ac, \Bc)$. By the Lebesgue decomposition theorem (see Theorem 4.3.2 in \cite{MR3098996}),
    $\mu$ can always be decomposed as $\mu = \mu_A + \mu_S$, where $\mu_A$ is a signed measure that is absolutely continuous with respect to $P_0$, $\mu_S$ is a signed measure that is singular with respect to $P_0$,
    and $\mu_A$ and $\mu_S$ are orthogonal.
    By the Jordan decomposition \eqref{Jordan_decomposition_theorem}, we decompose the signed measure
    $\mu_S = \alpha_+ \mu_{S,+} - \alpha_- \mu_{S,-}$
    into its positive and negative part $\mu_{S,+}$ and $\mu_{S,-}$. These two measures are orthogonal and $\alpha_+, \alpha_- \geq 0.$
    Then, $P_0 + a \mu = P_0 + \, a \mu_A + a \alpha_+ \mu_{S,+} - a \alpha_- \mu_{S,-}$ can be a probability measure only if $\alpha_- = 0$. This is because we can find a set $U$ such that $P_0(U) = \mu_A(U) = \mu_{S,+}(U) = 0$ and $\mu_{S,-}(U) = 1$.
    Therefore $(P_0 + a \mu)(U)
    = - a \alpha_- \mu_{S,-}(U) = - a \alpha_- \leq 0$.
    In the same way, $P_0 - a \mu$ can be a probability measure only if $\alpha_+ = 0$.
    Therefore, if $\mu \in R(a)$ for some $a>0$, then $\alpha_+ = \alpha_- = 0$, and $\mu = \mu_A$ is absolutely continuous with respect to $P_0$.

    \medskip

    Let $h$ be the density of $\mu$ with respect to $P_0$.
    Then
    \begin{align*}
        \frac{d(P_0 + t \mu)}{dP_0}
        = 1 + t \frac{d\mu}{dP_0} = 1 + t h.
    \end{align*}
    Note that $P_0 + t \mu$ is a signed measure integrating to $1$ if and only if $\int d\mu = \int h \, dP_0 = 0$.

    \medskip
    
    We now show that, for any $a > 0$, $\mu \in R(a)$ implies $\EssSup_{P_0} |h| \leq 1/a$. If $\mu \in R(a)$, then for any
    $A \in \Bc,$ $(P_0 + a \mu)(A) \geq 0$
    and $(P_0 - a \mu)(A) \geq 0$.
    Let us define the sets $A_+ := \{ x \in \Ac: \, 1 + a h(x) \geq 0\}$ and $A_- := \{ x \in \Ac: \, 1 - a h(x) \geq 0\}$.
    Let $A^C$ denote the complement of a set $A$.
    We have $(P_0 + a \mu)(A_+^C) = \int_{A_+^C} 1 + a h(x) dP_X(x) \leq 0$ since this is the integral of a negative function.
    Therefore $P_0(A_+^C) = 0$ and then $P_0(A_+) = 1$.
    Similarly, $(P_0 + a \mu)(A_-^C) = \int_{A_-^C} 1 - a h(x) dP_X(x) \leq 0.$ Hence, $P_0(A_-^C) = 0$ and $P_0(A_-) = 1$.

    \medskip

    Therefore, $P_0(A_+ \cap A_-) = 1$.
    This means that for $P_0$-almost every $x \in \Ac$, $1 + a h(x) \geq 0$ and $1 - a h(x) \geq 0$.
    Therefore, for $P_0$-almost every $x \in \Ac$, $|h(x)| \leq 1/a$.
    Therefore, $h$ is $P_0$-essentially bounded by $C := 1/a$.
    We have finally shown that $\mu\in R(a)$ implies $\EssSup_{P_0} |h| \leq 1/a$ and $\mu \in \Mc_{P_0}$, proving the first part of Proposition~\ref{prop:biggest_tangent_space}.

    \medskip

    Conversely, note that $\mu \in \Mc_{P_0}$ is a sufficient condition for $P_0 + t \mu$ to be a probability measure for all $t$ in a sufficiently small open neighborhood of $0$.

    \medskip
    
    We now show the second part of Proposition~\ref{prop:biggest_tangent_space}.
    Remember that $a_* := \sup \{ a > 0: \, \mu \in R(a) \} \in (0, + \infty]$.
    Let $(a_n)_{n \in \NN}$ be an increasing sequence of real numbers strictly smaller than $a_*$ and converging to $a_*$.
    For every positive integer $n$, we have $\mu \in R(a_n)$.
    Therefore, by the previous reasoning, $\EssSup_{P_0} |h| \leq 1/a_n$, meaning that $P_0(\{ x \in \Ac: \, |h(x)| \leq 1/a_n \}) = 1$.
    By a union bound, we obtain $P_0( \cap_{n \geq 0} \{ x \in \Ac: \, |h(x)| \leq 1/a_n \}) = 1$.
    Therefore $P_0( \{ x \in \Ac: \, |h(x)| \leq 1/a_* \}) = 1$, and by definition $\EssSup_{P_0}|h| \leq 1/ a_*$.

    \medskip

    We now show the reverse version of this inequality. Let $C > \EssSup_{P_0}|h|$.
    Then $P_0(\{x \in \Ac: \, |h(x)| \leq C\} = 1$.
    Hence, for any $t \in [-1/C, 1/C]$, and for $P_0$-almost every $x$, $-1 \leq t h(x) \leq 1$.
    Consequently, for any $t \in [-1/C, 1/C]$, and for $P_0$-almost every $x$,
    $1 + t h(x) \geq 0$ and $1 - t h(x) \geq 0$. For any $t \in [-1/C, 1/C]$, $P_0 + t \mu$ is a finite signed measure with a density that is non-negative $P_0$-almost everywhere and integrates to $1$.
    These are sufficient conditions for $P_0 + t \mu$ to be a probability measure on $\Ac$,
    proving $\mu \in R(1/C)$.
    Therefore, $1/C \leq a_*$ and thus $1/a_* \leq C$. This holds for any $C > 0$ such that $C > \EssSup_{P_0}|h|$, proving that $1/a_* \leq \EssSup_{P_0}|h|$.
    Together with the inequality $\EssSup_{P_0}|h| \leq 1/ a_*$, the claim $1/a_* = \EssSup_{P_0}|h|$ follows.
\end{proof}

\subsection{Examples of codivergences}

For any real $a\geq 0,$ set $a/0:=+\infty.$
For a non-negative function $\phi:[0,\infty)\to [0,\infty),$ we can define two codivergences. The first one will be referred to as covariance-type codivergence between three probability measures $P_0, P_1, P_2$ and is defined by
\begin{align}
    \Vphi(P_0 | P_1, P_2) := \int \phi\Big(\frac{dP_1}{dP_0}\Big)\phi\Big(\frac{dP_2}{dP_0}\Big) dP_0-
    \int \phi\Big(\frac{dP_1}{dP_0}\Big)dP_0 \int \phi\Big(\frac{dP_2}{dP_0}\Big) dP_0,
\label{eq:def:Vphi}
\end{align}
and the second one will be called correlation-type codivergence and is defined by
\begin{align}
    \Rphi(P_0 | P_1, P_2)
    := \frac{\Vphi(P_0 | P_1, P_2)}{\int \phi\big(\frac{dP_1}{dP_0}\big)dP_0 \int \phi\big(\frac{dP_2}{dP_0}\big) dP_0}
    =\frac{
    \int \phi\big(\frac{dP_1}{dP_0}\big)\phi\big(\frac{dP_2}{dP_0}\big) dP_0}{\int \phi\big(\frac{dP_1}{dP_0}\big)dP_0 \int \phi\big(\frac{dP_2}{dP_0}\big) dP_0} - 1,
\label{eq:def:Rphi}
\end{align}
if $P_1, P_2 \ll P_0.$
Otherwise, we define both $\Vphi(P_0 | P_1, P_2)$ and $\Rphi(P_0 | P_1, P_2)$ to be equal to $+ \infty$.

\medskip

Obviously, both codivergences $\Vphi$ and $\Rphi$ are symmetric in $P_1$ and $P_2$.
By Jensen's inequality we see that $\Vphi(P_0 | P_1, P_1)\geq 0$ and
$\Rphi(P_0 | P_1, P_1)\geq 0$.
If $\phi(1)=0,$ then $\Rphi(P_0|P_0,P_0) = +\infty.$
For $\phi(1)>0,$ the functions $\phi$ and $t\phi$ with positive scalar $t$ give the same codivergence $\Rphi$ and simply rescale $\Vphi$. Without loss of generality, we therefore can (and will) assume  that $\phi(1)=1.$

\medskip

We say that a function $f$ admits a second order Taylor expansion around $1$ if $f(1+y) = f(1) + y f'(1) +\frac{y^2}{2}f''(1) + o(y^2)$
for all $y$ in an open neighborhood of zero.
The following proposition is proved in Section~\ref{proof:prop:codiv}.

\begin{prop}
    Assume that $\phi(1) = 1$ and $\phi$ admit a second order Taylor expansion around $1$.
    Then the $\Vphi$ and the $\Rphi$ codivergences are codivergences in the sense of Definition \ref{def:codiv} with bilinear expansion domains $\Mc_{P_0}$ and bilinear maps $\phi'(1)^2\langle \mu, \wt \mu \rangle_{P_0},$ where
    \begin{align*}
        \langle \mu, \wt \mu \rangle_{P_0}
        := \int \frac{d \mu}{dP_0} d\wt \mu
        = \int h g \, dP_0,
        \quad \text{with} \ h = \frac{d\mu}{dP_0} \ \text{and} \ g = \frac{d\wt \mu}{dP_0}.
    \end{align*}
\label{prop:codiv}
\end{prop}

If $\nu$ is a measure dominating $P_0$, the bilinear map can be written as 
\begin{align}
    \langle \mu, \wt \mu \rangle_{P_0} = \int \frac{h g}{p_0} \, d\nu, \quad \text{for the densities} \ h = \frac{d\mu}{d\nu}, \ g = \frac{d\wt \mu}{d\nu}, \  p_0 = \frac{dP_0}{d\nu}.  
    \label{eq.P0_inner_prod_rewrite}
\end{align} 

A consequence of Proposition~\ref{prop:codiv} is the following: locally, all $\Vphi$ and $\Rphi$ codivergences (that satisfies the regularity conditions) define the same structure.
This scalar product is the nonparametric Fisher information metric.
The name originates from the identity \cite[Equation (8)]{holbrook2017nonparametric}
\begin{align}
    [I(\theta)]_{ij} = \int
    \frac{p_i(x | \theta) p_j(x | \theta)}{p(x | \theta)} d\nu(x),
\label{eq:Fisher_information}
\end{align}
where $[I(\theta)]_{ij}$ is the $(i,j)$-th entry of the Fisher information matrix for a parametric model of $\nu$-densities $p( \cdot | \theta)$ indexed by a finite dimensional parameter $\theta$ and
$p_i(x | \theta) := \partial p(x | \theta) / \partial\theta_i$. Equation~\eqref{eq:Fisher_information} and Equation~\eqref{eq.P0_inner_prod_rewrite} have the same structure.
One of the earliest reference to the nonparametric Fisher information metric is \cite{dawid1977further}.
The concept has been applied in several frameworks, such as computer vision~\cite{srivastava2007riemannian} or shape data analysis \cite{srivastava2016functional}. The geometry of the nonparametric Fisher information metric has been studied by \cite{chen2015geometric, holbrook2017nonparametric} in the context of Bayesian inference.

\medskip

An interesting subclass of codivergences is obtained by choosing 
$\phi_\alpha(x)
= x^\alpha.$
To ease the notation, we set
\begin{align}
    \Valpha := V_{\phi_\alpha} \quad \text{and} \ \     \Ralpha := R_{\phi_\alpha}.
\end{align}
Although the resulting codivergences seem related to the well-known R\'enyi divergence $(1-\alpha)^{-1} \log(\int p(x)^\alpha q(x)^{1-\alpha} \, d\nu(x)) $ between probability measures $P$ and $Q$ with densities $p$ and $q$ \cite{renyi1961measures}, the term
$\int (p_1(x)p_2(x))^\alpha p_0(x)^{1-2\alpha} d\nu(x)$
occurring in the definitions of $\Valpha$ and $\Ralpha$ is of a different nature.

\medskip

In the case $\alpha = 1$, that is, $\phi(x) = x,$ both notions of codivergence agree. Denoting by $p_0, p_1, p_2$ the respective $\nu$-densities of $P_0, P_1, P_2,$ where $\nu$ is a measure dominating $P_0$, the corresponding codivergence
\begin{align*}
    \chi^2(P_0 | P_1, P_2) :=
    \begin{cases} \displaystyle
    \int \frac{dP_1}{dP_0} dP_2 - 1
    = \int \frac{p_1 p_2}{p_0} d\nu - 1,
    & \text{if }
    P_1 \ll P_0 \text{ and } P_2 \ll P_0, \\
    + \infty, & \text{else,}
    \end{cases}
\end{align*}
will be called $\chi^2$-codivergence. The (usual) $\chi^2$-divergence is defined as $\chi^2(P,Q) := \int (dP/dQ-1)^2 dQ= \int (dP/dQ)^2 dQ-1$, if $P$ is dominated by $Q$ and $+\infty$ otherwise.
Therefore, the $\chi^2$-codivergence $\chi^2(P_0 | P_1, P_1)$
coincides with the usual $\chi^2$-divergence $\chi^2(P_1, P_0)$ for any $P_0$ and $P_1$.

\medskip

Another interesting codivergence is $\Ralpha$ with $\alpha = 1/2$. The resulting codivergence  
\begin{align}
    \rho(P_0 | P_1, P_2) :=
    \dfrac {\int \sqrt{p_1 p_2}  \, d\nu}{
    \int \sqrt{p_1 p_0} \, d\nu 
    \int \sqrt{p_2 p_0} \, d\nu} 
    - 1,
    \label{eq:def:rho}
\end{align}
is called Hellinger codivergence. We can (and will) define the Hellinger codivergence as $\int \sqrt{p_1 p_2} \, d\nu/(\int \sqrt{p_1 p_0} \, d\nu \int \sqrt{p_2 p_0} \, d\nu)$ whenever the denominator is positive.
This is considerably weaker than $P_1,P_2 \ll P_0,$ as it is only required that the support of $p_0$ intersects with non-zero $\nu$-mass the support of $p_1$ and the support of $p_2$.
%
%
%
Note that $\rho(P_0 | P_1, P_2)$ is independent of the choice of the dominating measure $\nu$
(and potentially $+\infty$ if the denominator is $0$). 

\medskip

The name Hellinger codivergence is motivated by the representation
\begin{align*}
    \rho(P_0 | P_1, P_2)
    = \frac{\alpha(P_1, P_2)}{\alpha(P_0, P_1) \alpha(P_0, P_2)} - 1,
\end{align*}
where $\alpha(P,Q) := \int \sqrt{pq} \, d\nu$ is the Hellinger affinity between two positive measures $P, Q$ with densities $p, q$ taken with respect to a common dominating measure.

\medskip

The $\chi^2$- and Hellinger codivergence are of interest as they can be used to control changes of expectation between probability measures, see Section 2.2 of \cite{2022lowerBoundsBV}.

\medskip

We always have 
\begin{align}
    \rho(P_0 | P_1, P_1) \leq \chi^2(P_0 | P_1, P_1).
    \label{eq.rho_chi2}
\end{align}
To see this, observe that H\"older's inequality with $p=3/2$ and $q=3$ gives for any non-negative function $f,$ $1=\int p_1\leq (\int f^{3/2} p_1)^{2/3}(\int f^{-3}p_1)^{1/3}.$ The choice $f=(p_0/p_1)^{1/3}$ yields $1\leq (\int \sqrt{p_1p_0})^2 \int p_1^2/p_0.$ Therefore $1 / (\int \sqrt{p_1p_0})^2 \leq \int p_1^2/p_0.$ Subtracting one on each side of this expression yields \eqref{eq.rho_chi2}.

\medskip 

Proposition~\ref{prop:codiv} implies that the $\chi^2$-codivergence and the Hellinger codivergence are codivergences with respective bilinear maps $\langle \mu, \wt \mu \rangle_{P_0}$ for the $\chi^2$-codivergence and $\langle \mu, \wt \mu \rangle_{P_0} / 4$ for the Hellinger codivergence.

\medskip

For the Hellinger codivergence, the expansion in Proposition~\ref{prop:codiv} can be generalized.
Assume that $P_0$ is dominated by some positive measure $\nu$. Define $\support(\mu) := \{ x \in \Ac: d\mu/d\nu(x) \neq 0\}$ for any signed measure $\mu$ dominated by $\nu.$ If $\mu_1$ and $\mu_2$ are signed measures dominated by $\nu$ such that (i) $\support(\mu_i) \cap \support(P_0)$ has a positive $\nu$-measure, and (ii) their densities $h_i$ are positive on $\support(\mu_i) \backslash \support(P_0)$, then
\begin{align}
    \rho(P_0 | P_0 + t \mu_1, P_0 + s \mu_2)
    &= \sqrt{t s} \int_{\support(P_0)^C} \sqrt{h_1 h_2} \, d\nu \nonumber \\
    &+ t s \int_{\support(P_0)} \frac{h_1 h_2}{2 p_0} \, d\nu
    + o(t^2 + s^2).
    \label{eq:general_expansion_rho}
\end{align}
Compared to Definition~\ref{def:codiv} (iii), there is thus an additional term for probability measures that have mass outside of the support of $P_0$. Consequently, this expansion cannot be linked to one local bilinear form and the mapping $(t, s) \in \Rb_+^2 \mapsto \rho(P_0 | P_0 + t \mu_1, P_0 + s \mu_2)$ is not differentiable at $(0,0)$. This is in line with Proposition~\ref{prop:biggest_tangent_space}: for perturbations $\mu$ that do not belong to $\Mc_{P_0}$, the measures $P_0 + t \mu$ cannot be probability measures for all $t$ in any open neighborhood of $0$. 

\medskip

The $\Ralpha$ codivergences admit convenient expressions for product measures and for exponential families.
The first proposition is proved in Section~\ref{proof:prop:phi_alpha_product}.

\begin{prop}
    Let $P_{j\ell}$ be probability measures
    for any $j=0, 1, 2$ and for any $\ell = 1, \dots, d$ satisfying $P_{1\ell}, P_{2\ell} \ll P_{0\ell}$.
    Then
    \begin{align*}
        \RalphaBigg{
        \bigotimes_{\ell=1}^d P_{0\ell}}{
        \bigotimes_{\ell=1}^d P_{1\ell}}{
        \bigotimes_{\ell=1}^d P_{2\ell}} 
        = \prod_{\ell=1}^d
        \Big( \Ralpha (P_{0\ell} | P_{1\ell}, P_{2\ell} ) + 1\Big) - 1.
    \end{align*}
    \label{prop:phi_alpha_product}
\end{prop}

\begin{prop}\label{prop.exponential_family}
Let $\Theta$ be a subset of a real vector space and let $(P_\theta:\theta\in \Theta)$ be an exponential family with $\nu$-densities $p_\theta(x)=h(x)\exp(\theta^\top T(x)-A(\theta))$ for some dominating measure $\nu$.
Then, for any $\theta_0,\theta_1,\theta_2\in \Theta$
satisfying
\begin{align}
    \theta_0 + \alpha\big(\theta_1+\theta_2-2\theta_0\big), \,
    \theta_0 + \alpha\big(\theta_1-\theta_0\big), \,
    \theta_0 + \alpha\big(\theta_2-\theta_0\big) \in \Theta,
    \label{eq:cond:theta_expfamily}
\end{align}
we have
\begin{align*}
    \Ralpha(P_{\theta_0}|P_{\theta_1},P_{\theta_2})=
    &\exp\Big(A\big(\theta_0+\alpha\big(\theta_1+\theta_2-2\theta_0\big)\big)
    - A\big(\theta_0+\alpha\big(\theta_1-\theta_0\big)\big) \\
    &\quad\quad - A\big(\theta_0+\alpha\big(\theta_2-\theta_0\big)\big)
    + A\big(\theta_0\big)\Big)-1.
\end{align*}
\end{prop}


This proposition is proved in Section~\ref{proof:prop.exponential_family}. \eqref{eq:cond:theta_expfamily} is satisfied if $\Theta$ is a vector space or if $0<\alpha \leq 1$ and $\Theta$ is convex. In the case of the Gamma distribution the parameter space is $\Theta = (-1, +\infty) \times (- \Rb)$ and in this case the constraints in  \eqref{eq:cond:theta_expfamily} are sufficient and necessary for the statement of Proposition~\ref{prop.exponential_family} to hold, see Section~\ref{sec:computation_Gamma_distr} for details.

\medskip

For the most common families of distributions, closed-form expressions for the $\Ralpha(P_{\theta_0}|P_{\theta_1},P_{\theta_2})$ codivergences are reported in Table \ref{tab:explicit_expressions}. Derivations for these expressions are given in Section \ref{sec.explicit}.
This section also contains expressions for the Gamma distribution.
As mentioned before, these codivergences quantify to which extent the measures $P_1$ and $P_2$ represent different directions around $P_0.$
The explicit formulas show this in terms of the parameters and reveal significant similarity between the different families.
For the multivariate normal distribution the $\Ralpha$ codivergence vanishes if and only if the vectors $\theta_1-\theta_0$ and $\theta_2-\theta_0$ are orthogonal.


\begin{table}[htb]
    \centering
    \renewcommand{\arraystretch}{1.5}
    \resizebox{\textwidth}{!}{%
    \begin{tabular}{c|c}
    \text{distribution}
    & $\Ralpha(P_0 | P_1, P_2)$
    \\ \hline
    $P_j=\Nc(\theta_j, \sigma^2 I_d),$ 
    & \multirow{2}{*}{ $\exp\Big(\alpha^2 \dfrac{\langle \theta_1 - \theta_0, \theta_2 - \theta_0 \rangle}{\sigma^2}\Big) -1$ } \\ 
    $\theta_j \in \Rb^d,$ $\sigma > 0$ &
    \\ \hline
    $P_j = \otimes_{\ell=1}^d \Pois(\lambda_{j\ell}),$ 
    & \multirow{2}{*}{ $\exp\Big(\sum_{\ell=1}^d \lambda_{0\ell}^{1 - 2\alpha}
    \big( \lambda_{1\ell}^\alpha - \lambda_{0\ell}^{\alpha} \big)
    \big( \lambda_{2\ell}^\alpha - \lambda_{0\ell}^{\alpha} \big)
    \Big) - 1$} \\
    $\lambda_{j\ell}>0$
    &
    \\  \hline
    $P_j = \otimes_{\ell=1}^d \Exp(\beta_{j\ell}),$
    & $\displaystyle \prod_{\ell=1}^d
    \dfrac{
    \big(\beta_{0\ell}
    + \alpha (\beta_{1\ell} - \beta_{0\ell}) \big)
    \big(\beta_{0\ell}
    + \alpha (\beta_{2\ell} - \beta_{0\ell}) \big)
    }{
    \beta_{0\ell} \big(\beta_{0\ell}
    + \alpha (\beta_{1\ell} + \beta_{2\ell} - 2 \beta_{0\ell}) \big)
    }
    - 1,$
    \\ 
    $\beta_{j\ell}>0$
    & if all the involved quantities are positive,
    and $+\infty$ else
    \\ \hline
    \multirow{2}{*}{$P_j = \otimes_{\ell=1}^d \Ber(\theta_{j\ell}),$}
    & \multirow{2}{*}{ $\displaystyle \prod_{\ell=1}^d
    \dfrac{\theta_{0\ell}^{1 - 2\alpha} \theta_{1\ell}^\alpha \theta_{2\ell}^\alpha
    + (1 - \theta_{0\ell})^{1 - 2 \alpha} (1 - \theta_{1\ell})^\alpha (1 - \theta_{2\ell})^\alpha}{
    r(\theta_{0\ell}, \theta_{1\ell}) r(\theta_{0\ell}, \theta_{2\ell})
    }
    - 1,$}
    \\
    & 
    \\
    $\theta_{j\ell}\in (0,1)$ &
    $r(\theta_0,\theta_1) := \theta_0^{1 - \alpha} \theta_1^\alpha
    + (1 - \theta_0)^{1 - \alpha} (1 - \theta_1)^\alpha$ \\
\end{tabular}
}
\vspace{0.2em}
\caption{Closed-form expressions for the $\Ralpha$ codivergence for some parametric distributions. Proofs can be found in Section \ref{sec.explicit}.}
\label{tab:explicit_expressions}
\end{table}

\section{Divergence matrices}
\label{sec:div_matrices}

\begin{defi}
    Let $M\geq 1.$ For a given codivergence $D( \cdot | \cdot, \cdot)$ on a space $\Xc \subset E$ and $u, v_1, \dots, v_M$ elements of $\Xc$, we define the divergence matrix $D(u | v_1, \dots, v_M)$ as the $M \times M$ matrix with $(j,k)$-th entry
    $D(u | v_1, \dots, v_M)_{j,k} := D(u | v_j, v_k)$, for all $1 \leq j,k \leq M$.
\end{defi}

If $v_1, \dots, v_M$ are all in a neighborhood of $u$, the divergence matrix $D$ can be related to the Gram matrix of the bilinear form
$\langle \cdot, \cdot \rangle_{u}$.
Formally, for $\t = (t_1, \dots, t_M) \in \Rb^M$ such that for any $i = 1, \dots, M, \, u + t_i h_i \in \Xc$, we have
\begin{align*}
    D(u | u + t_1 h_1, \dots, u + t_M h_M)
    = \t \Gb_{u}  \t^\top + o(\|\mathbf{t}\|^2),
\end{align*}
with Gram matrix $\Gb_{u}
:= (\langle h_i, h_j \rangle_{u})_{1 \leq i,j \leq M}$.

\medskip

Based on the codivergences $\Vphi(P_0|P_1,P_2), \Rphi(P_0|P_1,P_2),$ one can now define corresponding $M\times M$ divergence matrices with $(j,k)$-th entry
\begin{align}
    &\Vphi(P_0|P_1,\ldots,P_M)_{j,k}
    :=\Vphi(P_0|P_j,P_k) \\
    &\hspace{1.5cm} = 
    \int \phi\Big(\frac{dP_j}{dP_0}\Big)\phi\Big(\frac{dP_k}{dP_0}\Big) dP_0-\int \phi\Big(\frac{dP_j}{dP_0}\Big)dP_0 \int \phi\Big(\frac{dP_k}{dP_0}\Big) dP_0, \notag
\end{align}
and
\begin{align}
    \hspace{-0.3cm} \Rphi(P_0|P_1,\ldots,P_M)_{j,k}
    :=\Rphi(P_0|P_j,P_k) = \dfrac{
    \int \phi\big(\frac{dP_j}{dP_0}\big)\phi\big(\frac{dP_k}{dP_0}\big) dP_0}{\int \phi\big(\frac{dP_j}{dP_0}\big)dP_0 \int \phi\big(\frac{dP_k}{dP_0}\big) dP_0} - 1,
\end{align} 
provided that $P_1,\ldots,P_M\ll P_0.$
The codivergence matrices are linked by the relationship
\begin{align}
    \Rphi(P_0|P_1,\ldots,P_M)
    = D \cdot \Vphi(P_0|P_1,\ldots,P_M) \cdot D,
\label{eq:link_matrices_phi_RV}
\end{align}
where $D$ denotes the $M\times M$ diagonal matrix with $j$-th diagonal entry $1/\int \phi\big(\frac{dP_j}{dP_0}\big) dP_0,$ $j=1,\ldots,M.$

\medskip

Similarly as $\Cov(X_1, X_2)$ can denote either the covariance between the random variables $X_1$ and $X_2$ or the $2\times 2$ covariance matrix of the random vector $(X_1,X_2),$ the expressions $\Vphi(P_0|P_1,P_2)$ and $\Rphi(P_0|P_1,P_2)$ can also denote either codivergences or $2\times 2$ divergence matrices. Within the context, it is always clear which of the two interpretations is meant. 

\medskip

The divergence matrices with function $\phi_\alpha(x) := x^\alpha$ are denoted by $\Valpha(P_0 | P_1, \dots, P_M)$ and $\Ralpha(P_0 | P_1, \dots, P_M)$.
Similarly,
%
%
%
the $\chi^2$-divergence matrix $\chi^2(P_0 | P_1, \dots, P_M)$
and the Hellinger affinity matrix $\rho(P_0 | P_1, \dots, P_M)$ are the $M\times M$ divergence matrices of the $\chi^2$-codivergence and the Hellinger codivergence with $(j,k)$-th entry
\begin{align*}
    &\chi^2(P_0 | P_1, \dots, P_M)_{j,k}
    := \int \frac{dP_j}{dP_0} dP_k - 1
    \text{, \ \  and,} \\
    &\rho(P_0 | P_1, \dots, P_M)_{j,k}
    := \frac {\int \sqrt{p_j p_k}  \, d\nu}
    {\int \sqrt{p_j p_0  }\, d\nu  \int \sqrt{p_k p_0} \, d\nu} -1,
\end{align*}
for all $1 \leq j, k \leq M$.
As in the previous section, the condition for finiteness of the Hellinger codivergence matrix is weaker than for general $\Rphi$ and $\Vphi$ codivergences. Instead of domination $P_1, \dots, P_M \ll P_0$, it is only required that the integrals $\int \sqrt{p_j p_0} \, d\nu$ are positive, for some dominating measure $\nu$ and $p_j := dP_j / d\nu$.
By \eqref{eq.P0_inner_prod_rewrite},
the local Gram matrix of the $\chi^2$-divergence matrix at a distribution $P_0$ is
$\Gb_{P_0} := \big[\int \frac{h_i h_j}{p_0} d\nu \big]_{1 \leq i,j \leq M},$
and the local Gram matrix of the Hellinger divergence matrix is $\Gb_{P_0} / 4$.

\medskip

Let $\Phi(X) := (\phi(dP_1/dP_0(X)), \dots, \phi(dP_M/dP_0(X)))^\top$ denote the random vector containing the likelihood ratios of the $M$ measures. Since $\Cov(U,V) = E[UV] - E[U]E[V],$ we have 
\begin{align}
    \Vphi(P_0|P_1,\ldots,P_M) = \Cov_{P_0}\big(\Phi(X)\big),
\label{eq:expression_covPhi_V}
\end{align}
where the covariance is computed with respect to the distribution $P_0$ as indicated by the subscript $P_0.$
Moreover, we have 
\begin{align*}
    \v^\top \Vphi(P_0|P_1,\ldots,P_M) \v
    &= \Var_{P_0} \big(\v^\top \Phi(X)\big).
\end{align*}
Applying Equation~\eqref{eq:link_matrices_phi_RV} yields moreover
\begin{align}
    \Rphi(P_0|P_1,\ldots,P_M) = D \Cov_{P_0}\big(\Phi(X)\big) D.
\label{eq:expression_covPhi_R}
\end{align}

This shows that $\Vphi(P_0|P_1,\ldots,P_M)$ and $\Rphi(P_0|P_1,\ldots,P_M)$ can be interpreted as covariance matrices and are therefore symmetric and positive semi-definite. Applying the Taylor expansion to the likelihood ratios in the previous identities provides a direct way of recovering the local Gram matrix associated to the nonparametric Fisher information metric.

\medskip

In a next step, we state a more specific identity for the $\chi^2$-divergence matrix. To do so, we first extend the usual notion of the $\chi^2$-divergence to the case where the first argument is a signed measure. Let $\mu$ be a finite signed measure and $P$ be a probability measure defined on the same measurable space $(\Omega, \mathcal{A})$. We define the $\chi^2$-divergence of $\mu$ and $P$ by
\begin{align}
    \chi^2(\mu,P) := \begin{cases}
    \displaystyle \int \Big( \frac{d\mu}{dP} - \mu(\Omega)\Big)^2 \, dP,
    & \text{ if } \mu\ll P, \\
    + \infty & \text{ else.}
    \end{cases} 
    \label{eq.chi_divergence_signed_measure}
\end{align}
Here, $d\mu/dP$ denotes the Radon-Nikodym derivative of the signed measured $\mu$ with respect to $P$ (defined e.g. in Theorem 4.2.4 in \cite{MR3098996}).
This definition of $\chi^2(\mu,P)$ generalizes the case where $\mu$ is a probability measure and allows us to rewrite the $\chi^2$-divergence matrix as
\begin{align}
    \hspace{-0.55em}
    \v^\top \hspace{-0.1em} \chi^2(P_0 | P_1, \dots, P_M) \v
    \hspace{-0.2em}
    &= \hspace{-0.2em} \int \hspace{-0.2em}
    \Big(\sum_{j=1}^M \Big(\frac{dP_j}{dP_0}-1\Big) v_j \Big)^2 dP_0
    \hspace{-0.1em}
    = \chi^2\Big(\sum_{j=1}^M v_j P_j, P_0\Big),
    \label{eq.chi2_int_representation}
\end{align}
with $\sum_{j=1}^M v_j P_j$ the mixture (signed) measure of $P_1, \dots, P_M.$
Similarly, for the Hellinger divergence matrix it can be checked that 
\begin{align}
    \v^\top \rho(P_0 | P_1,\dots,P_M) \v
    = \int \bigg( \sum_{j=1}^M \Big(\frac{\sqrt{p_j}}{\int \sqrt{p_j p_0}\, d\nu}-\sqrt{p_0} \Big) v_j \bigg)^2\, d\nu.
    \label{eq.Hell_aff_pos_def}
\end{align}

Writing $\Rank(A)$ for the rank of a matrix $A$ and $\Rank(x_1, \dots, x_n)$ for the dimension of the linear span of $n$ elements $x_1, \dots, x_n$ in a vector space $E$, we will now derive an identity for the rank of divergence matrices.

\begin{prop}
Let $M \geq 1$, and let $P_0, P_1, \dots, P_M$ be $(M+1)$ probability distributions.

\vspace{0.3em}

\begin{itemize}    
    \item[(i)] Assume that $P_1, \dots, P_M \ll P_0$. Then for any non-negative function $\phi:[0,\infty)\to [0,\infty)$ such that $\phi(1)=1$, we have
    \begin{align*}
        \vspace{-0.3em}
        \Rank(\Rphi(P_0 | P_1, \dots, P_M))
        &= \Rank(\Vphi(P_0 | P_1, \dots, P_M)) \\
        &= \Rank \bigg(1, \phi \, \circ \, \frac{dP_{1}}{dP_0}, \dots, \phi \, \circ \, \frac{dP_M}{dP_0} \bigg) - 1,
    \end{align*}
    where functions are considered as elements of the vector space $L^1(\Ac, \Bc, P_0)$, that is, linear independence is considered $P_0$-almost everywhere.

    \vspace{0.3em}

    
    \item[(ii)] Let $\nu$ be a common dominating measure of $P_0, \dots, P_M$. Assume that $\forall j = 1, \dots, M, \,
    \int p_j p_0 d\nu > 0$ with $p_j := dP_j / d\nu$.
    Then we have $$\Rank(\rho(P_0 | P_1, \dots, P_M)) = \Rank(\sqrt{p_0}, \sqrt{p_1}, \dots, \sqrt{p_M}) - 1,$$
    where functions are considered as elements of the vector space $L^1(\Ac, \Bc, \nu)$.
\end{itemize}
\label{prop:div_matrices}
\end{prop}

Statement (ii) is not a consequence of (i) with $\phi(x) = x^{1/2}.$ Indeed, (i) relies on likelihood ratios assuming that the measures $P_1, \dots, P_M$ are dominated by $P_0,$ while (ii) only requires that each of the probability measures $P_1, \dots, P_M$ has a common support with $P_0$ of positive $P_0$-measure. The proof of (ii) exploits the specific property~\eqref{eq.Hell_aff_pos_def} of the Hellinger divergence.

\medskip

Proposition~\ref{prop:div_matrices} applied to $\phi(x) = x$ shows that whenever $P_0$ is a linear combination of $P_1, \dots, P_M$, then $\Rank(1, dP_{1}/dP_0, \dots, dP_{M}/dP_0) < M + 1$ and $\Rank(\chi^2(P_0|P_1, \dots, P_M)) < M,$ which means that the $\chi^2(P_0|P_1, \dots, P_M)$ divergence matrix is singular.
Similarly, whenever $\sqrt{p_0}$ is a linear combination of $\sqrt{p_1}, \dots, \sqrt{p_M}$, the Hellinger divergence matrix is singular.

\begin{proof}[Proof of Proposition~\ref{prop:div_matrices}]

We first prove (i).
Since $D$ is an invertible matrix, a direct consequence of Equation~\eqref{eq:link_matrices_phi_RV} is $\Rank(\Rphi(P_0 | P_1, \dots, P_M))
= \Rank(\Vphi(P_0 | P_1, \dots, P_M)).$
Applying Equation~\eqref{eq:expression_covPhi_V} and then Lemma~\ref{lemma:rank_CovZ}, we obtain that
$r := \Rank(\Vphi(P_0 | P_1, \dots, P_M))
= \Rank(\Cov_{P_0}(Z_1, \dots, Z_M))
= \Rank(Z_1 - E_{P_0} Z_1, \dots, Z_M - E_{P_0} Z_M)$,
where $Z_j := \phi(dP_j/dP_0(X))$ for $j=1, \dots, M$ and $E_{P_0}$ denotes the expectation with respect to $P_0$.
The random vectors $Z_1 - E_{P_0} Z_1, \dots, Z_M - E_{P_0} Z_M$ are centered and therefore linearly independent of the (constant) random variable $Z_0 := 1 = \phi(dP_0 / dP_0(X))$.
Therefore,
\begin{align*}
    r &= \Rank(Z_1 - E_{P_0} Z_1, \dots, Z_M - E_{P_0} Z_M) \\
    &= \Rank(1, Z_1 - E_{P_0} Z_1, \dots, Z_M - E_{P_0} Z_M) - 1 \\
    &= \Rank(Z_0, Z_1, \dots, Z_M) - 1.
\end{align*}

By Lemma~\ref{lemma:rank_max_lincombin}, $r$ is the highest integer such that there exists $i_1, \dots, i_r \in \{0, \dots, M\}$ with $(Z_{i_1}, \dots, Z_{i_r})$ linearly independent random variables $P_0$-almost surely.

Using the definition of the $Z_j$ and $X \sim P_0$, the random variables $\{Z_{i_1}, \dots, Z_{i_r}\}$ are linearly independent $P_0$-almost surely if and only if
$P_0 \big(\sum_{j=1}^r a_j \phi(dP_{i_j}/dP_0(X)) = 0 \big)
= 1$ implies $a_0=\ldots=a_r=0$.
This is the case if and only if 
the functions
$\{\phi \, \circ \, dP_{i_1}/dP_0, \dots,
\phi \, \circ \, dP_{i_r}/dP_0\}$\
are linearly independent $P_0$-almost everywhere, proving
$\Rank(Z_{i_1}, \dots, Z_{i_r})
= \Rank(\phi \, \circ \, dP_{i_1}/dP_0, \dots,
\phi \, \circ \, dP_{i_r}/dP_0)$.



\medskip

Before proving (ii) in full generality, we first show that $\Rank(\rho(P_0 | P_1, \dots, P_M)) = M$ if and only if all the $M+1$ functions $\sqrt{p_0}, \dots, \sqrt{p_M}$ are linearly independent $\nu$-almost everywhere.
The matrix is singular if and only if there exists a non-null vector $v$ such that
$\sum_{j=1}^M \frac{v_j \sqrt{p_j}}{\int \sqrt{p_j p_0}\, d\nu}
= \sum_{j=1}^M v_j \sqrt{p_0}$ $\nu$-almost everywhere.
This is the case if and only if there are numbers $w_0,\dots,w_M,$ that are not all equal to zero, satisfying $\sum_{j=0}^M w_j\sqrt{p_j}=0,$ $\nu$-almost everywhere. To verify the more difficult reverse direction of this equivalence, it is enough to observe that $\sum_{j=0}^M w_j\sqrt{p_j}=0$ implies $w_0=- \sum_{j=1}^M w_j \int \sqrt{p_j p_0}\, d\nu$ and thus, taking $v_j=w_j \int \sqrt{p_j p_0}\, d\nu$ yields $\sum_{j=1}^M \frac{v_j \sqrt{p_j}}{\int \sqrt{p_j p_0}\, d\nu}
= \sum_{j=1}^M v_j \sqrt{p_0}.$

\medskip

We now show the general case of (ii). For an $n\times n$ matrix $A$ and index sets $I,J \subset \{1, \dots, n\}$, the submatrix $A_{I,J}$ defines the submatrix consisting of the rows $I$ and the columns $J$. If $I=J$, $A_{I,I}$ is called a principal submatrix of the matrix~$A$. Let $r$ be an integer in $\{1, \dots, M\}$. By Lemma \ref{lemma:rank_principal_minors},
$$r = \Rank(\rho(P_0 | P_1, \dots, P_M))$$
if and only if
$$\begin{array}{c}
\rho(P_0 | P_1, \dots, P_M)
\text{ has an invertible principal submatrix of size } r \\
\text{ and all principal submatrix of size } r+1 
\text{ of } \rho(P_0 | P_1, \dots, P_M) 
\text{ are singular}
\end{array}
$$
if and only if (using the fact that the principal submatrices of $\rho(P_0 | P_1, \dots, P_M)$ of size $k$ are exactly the matrices of the form $\rho(P_0 | P_{i_1}, \dots, P_{i_k})$ for some $i_1, \dots, i_k \in \{1, \dots M\}$)
\begin{align*}
    r = \max_{} \{k=1, \dots, M: \,
    \exists i_1, \dots, i_k \in \{1, \dots, M\}, \,
    \rho(P_0 | P_{i_1}, \dots, P_{i_k}) \text{ is invertible} \}
\end{align*}
if and only if (using the case of full rank that was proved before)
\begin{align*}
    r = \max_{} &\bigg\{k=1, \dots, M: \,
    \exists i_1, \dots, i_k \in \{1, \dots, M\}, \\
    &\sqrt{p_0}, \sqrt{p_{i_1}}, \dots, \sqrt{p_{i_r}}
    \text{ are linearly independent} \bigg\}
\end{align*}
if and only if 
$r = \Rank(\sqrt{p_0}, \sqrt{p_1}, \dots, \sqrt{p_M}) - 1.$

\end{proof}


\section{Data processing inequality for the \texorpdfstring{$\chi^2$}{chi2}-divergence matrix}
\label{sec:data_processing}

In a parametric statistical model $(Q_\theta)_{\theta \in \Theta}$, it is assumed that the statistician observes a random variable $X$ following one of the distributions $Q_\theta$ for some $\theta \in \Theta$. If we transform $X$ to obtain a new variable $Y$, then $Y$ follows the distribution $P_\theta := K Q_\theta$ for some Markov kernel $K$. When $\theta$ is unknown but the Markov kernel $K$ is known and independent of $\theta$, this means that the new statistical model is $(P_\theta := K Q_\theta, \, \theta \in \Theta)$.
As in the usual case for the $\chi^2$-divergence, it is natural to think that such a transformation cannot increase the amount of information present in the model.
In our more general framework, such an inequality still holds and is presented in the following data processing inequality.
\begin{thm}[Data processing / entropy contraction]\label{thm.data_processing}
    If $K$ is a Markov kernel and $Q_0,\dots,Q_M$ are probability measures such that $Q_0$ dominates $Q_1,\dots,Q_M,$ then,
    \begin{align*}
        \chi^2\big(K Q_0 | KQ_1,\dots, KQ_M\big) 
        \leq \chi^2\big(Q_0 | Q_1,\dots, Q_M\big),
    \end{align*}
    where $\leq$ denotes the partial order on the set of positive semi-definite matrices.
\end{thm}
In particular, the $\chi^2$-divergence matrix is invariant under invertible transformations.
The rest of this section is devoted to the proof of Theorem \ref{thm.data_processing}.
First, we generalize the well-known data-processing inequality for the $\chi^2$-divergence to the case \eqref{eq.chi_divergence_signed_measure}, where one measure is a finite signed measure and use afterwards Equation~\eqref{eq.chi2_int_representation}.

\medskip

The $\chi^2(\mu,P)$-divergence with a signed measure can be computed from the usual $\chi^2$-divergence between probability measures by the following relationship
\begin{lemma}
    Assume that $\mu \ll P$.
    Let $\mu=\alpha_+\mu_+-\alpha_-\mu_-$ be the Jordan decomposition \eqref{Jordan_decomposition_theorem} of $\mu$ with $\alpha_+, \alpha_- \geq 0$ and $\mu_+,$ $\mu_-$ orthogonal probability measures.
    Then 
    $$\chi^2(\mu,P) = \alpha_+^2 \chi^2\big(\mu_+,P\big)
    + \alpha_-^2 \chi^2 \big(\mu_- , P\big) + 2\alpha_+ \alpha_-.$$
    \label{lemma:chi2_signed_measure}
\end{lemma}
\begin{proof}
    Observe that $\displaystyle \alpha_+^2 \chi^2\big(\mu_+,P\big)
    +\alpha_-^2 \chi^2\big(\mu_-,P\big)
    +2\alpha_+\alpha_- 
    =\int \Big(\alpha_+\Big(\frac{d\mu_+}{dP}-1\Big)-\alpha_-\Big(\frac{d\mu_-}{dP}-1\Big)\Big)^2 \, dP
    = \int \Big(\frac{d\mu}{dP}-\mu(\Omega)\Big)^2 \, dP
    =\chi^2(\mu,P).$
\end{proof}

\medskip

\begin{lemma}
    If $\mu$ is a finite signed measure, $P$ is a probability measure and both measures are defined on the same measurable space, then, for any Markov kernel $K,$ the data-processing inequality
    \begin{align*}
        \chi^2(K\mu,KP)
        \leq \chi^2(\mu,P)
    \end{align*}
    holds.
    \label{lem.data_processing_signes_measure}
\end{lemma}

\begin{proof}
We can assume that $\mu \ll P,$ since otherwise the right-hand side of the inequality is $+\infty$ and the result holds. In particular, $\mu \ll \nu$ for a positive measure $\nu$ implies that $K\mu \ll K\nu.$ Indeed, if $K\nu(A)=0$ for a given measurable set $A$, then, $\int K(A,x) \, d\nu(x)=0,$ implying $K(A,\cdot)=0$ $\nu$-almost everywhere. Since $\mu\ll \nu,$ the equality also holds $\mu$-almost everywhere and so $K\mu(A)=\int K(A,x)d\mu(x)=0,$ proving $K\mu \ll K\nu.$ By the Jordan decomposition \eqref{Jordan_decomposition_theorem}, there exist orthogonal probability measures $\mu_+,$ $\mu_-$ and non-negative real numbers $\alpha_+,\alpha_-,$ such that $\mu=\alpha_+\mu_+-\alpha_-\mu_-$ and $\mu(
\Omega)=\alpha_+-\alpha_-.$ Thus, $K\mu=\alpha_+K\mu_+-\alpha_-K\mu_-.$ Observe that 
\begin{align*}
    \int \Big(\frac{dK\mu_+}{dKP}-1\Big)\Big(\frac{dK\mu_-}{dKP}-1\Big) \, dKP
    &= \int \bigg(  \frac{dK\mu_+}{dKP} \frac{dK\mu_-}{dKP}
    - \frac{dK\mu_-}{dKP}
    - \frac{dK\mu_+}{dKP}
    + 1 \bigg) \, dKP \\ 
    &= \int \frac{dK\mu_+}{dKP} \frac{dK\mu_-}{dKP} \, dKP-1
    \geq -1.
\end{align*}
Because $\mu_+$ and $\mu_-$ are orthogonal, we similarly find that 
\begin{align*}
    \int \Big(\frac{d\mu_+}{dP}-1\Big)\Big(\frac{d\mu_-}{dP}-1\Big) \, dP =-1.
\end{align*}
Using the data-processing inequality for the $\chi^2$ divergence of probability measures twice, $K\mu=\alpha_+K\mu_+-\alpha_-K\mu_-$ and $\mu(\Omega)=\alpha_+-\alpha_-,$ we get
\begingroup \allowdisplaybreaks
\begin{align*}
    \chi^2(K\mu, KP)
    &=
    \int \Big(\frac{dK\mu}{dKP}-\mu(\Omega)\Big)^2 \, dKP \\
    &= 
    \int \Big(\alpha_+\Big(\frac{dK\mu_+}{dKP}-1\Big)-\alpha_-\Big(\frac{dK\mu_-}{dKP}-1\Big)\Big)^2 \, dKP \\
    &= \alpha_+^2 \chi^2\big(K\mu_+,KP\big)
    +\alpha_-^2 \chi^2\big(K\mu_-,KP\big) \\
    &- 2\alpha_+\alpha_-\int \Big(\frac{dK\mu_+}{dKP}-1\Big)\Big(\frac{dK\mu_-}{dKP}-1\Big) \, dKP \\
    &\leq \alpha_+^2 \chi^2\big(\mu_+,P\big)
    +\alpha_-^2 \chi^2\big(\mu_-,P\big)
    +2\alpha_+\alpha_- =\chi^2(\mu,P),
\end{align*}
\endgroup
by Lemma~\ref{lemma:chi2_signed_measure}.

\end{proof}

We can now complete the proof of Theorem \ref{thm.data_processing}.

\begin{proof}[Proof of Theorem \ref{thm.data_processing}]
Let $v=(v_1,\ldots,v_M)^\top \in \Rb^M.$ Then, $\sum_{j=1}^M v_jQ_j$ is a finite signed measure dominated by $Q_0$. Using \eqref{eq.chi_divergence_signed_measure} and the previous lemma,
\begingroup \allowdisplaybreaks
\begin{align*}
    v^T \chi^2(K Q_0 | KQ_1, \dots, K Q_M) v
    &=
    \int \bigg(\sum_{j=1}^M v_j\Big(\frac{dKQ_j}{dKQ_0}-1\Big)\bigg)^2 \, dKQ_0 \\
    &= \chi^2\Bigg(K\Big(\sum_{j=1}^M v_jQ_j\Big),KQ_0\Bigg) \\
    &\leq \chi^2\Big(\sum_{j=1}^M v_jQ_j,Q_0\Big) \\
    &= \int \bigg(\sum_{j=1}^M v_j\Big(\frac{dQ_j}{dQ_0}-1\Big)\bigg)^2 \, dQ_0 \\
    &= v^T \chi^2(Q_0 | Q_1, \dots, Q_M) v.
\end{align*}
\endgroup
Since $v$ was arbitrary, this completes the proof.

\end{proof}

A Markov kernel $K$ implies by definition that for every fixed $x,$ $A\mapsto K(A,x)$ is a probability measure.
We now provide a simpler and more straightforward proof for Theorem \ref{thm.data_processing} without using Lemma \ref{lem.data_processing_signes_measure}, under the additional common domination assumption:
\begin{align}
    \text{There exists a measure } \mu,
    \text{ such that } \forall x \in \Omega,
    K(x, \, \cdot \,) \ll \mu.
    \label{eq:assump:common_domination}
\end{align}

\begin{proof}[Simpler proof of Theorem \ref{thm.data_processing} under the additional assumption \eqref{eq:assump:common_domination}]
\ \newline
Because of the identity
$v^\top \chi^2(Q_0 | Q_1 , \dots , Q_M)v
= \int(\sum_{j=1}^M v_j(dQ_j/dQ_0-1))^2 dQ_0,$ it is enough to prove that for any arbitrary vector $v=(v_1,\dots,v_M)^\top$,
\begin{align}
    \int \bigg( \sum_{j=1}^M v_j \Big(\frac{dKQ_j}{dKQ_0}-1\Big) \bigg)^2 \, dKQ_0
    \leq \int \bigg( \sum_{j=1}^M v_j \Big(\frac{dQ_j}{dQ_0}-1\Big) \bigg)^2 \, dQ_0.
    \label{eq.entr_contr_to_show}
\end{align}
Let $\nu$ be a dominating measure for $Q_0,\dots,Q_M$ and recall that by the additional assumption \eqref{eq:assump:common_domination}, for any $x,$ the measure $\mu$ is a dominating measure for the probability measure $A\mapsto K(A,x).$
Write $q_j$ for the $\nu$-density of $Q_j.$ Then, $dKQ_j(y)=\int_X k(y,x) q_j(x) \, d\nu(x) \, d\mu(y)$ for $j=1,\dots,M$ and a suitable non-negative kernel function $k$ satisfying $\int k(y,x) \, d\mu(y)=1$ for all $x.$ Applying the Cauchy-Schwarz inequality, we obtain
\begin{align*}
    \bigg( \sum_{j=1}^M v_j \Big(\dfrac{dKQ_j}{dKQ_0}(y)-1\Big) \bigg)^2
    &=
    \bigg( \dfrac{\int k(y,x) [\sum_{j=1}^M v_j (q_j(x)-q_0(x))] \, d\nu(x)}{\int k(y,x')q_0(x') \, d\nu(x')} \bigg)^2 \\
    &\leq 
    \dfrac{\int k(y,x)\big( \sum_{j=1}^M v_j
    \frac{(q_j(x)-q_0(x))}{q_0(x)}\big)^2 q_0(x) \, d\nu(x)}{\int k(y,x')q_0(x') \, d\nu(x')}.
\end{align*}
Inserting this in \eqref{eq.entr_contr_to_show}, rewriting $dKQ_0(y)=\int_X k(y,x) q_0(x) \, d\nu(x) \, d\mu(y),$ interchanging the order of integration using Fubini's theorem, and applying $\int k(y,x) \, d\mu(y)=1,$ yields
\begin{align*}
    &\int \bigg( \sum_{j=1}^M v_j \Big(\frac{dKQ_j}{dKQ_0}-1\Big) \bigg)^2 \, dKQ_0\\
    &\leq 
    \iint k(y,x) \bigg(\sum_{j=1}^M v_j \frac{(q_j(x)-q_0(x))}{q_0(x)}\bigg)^2 q_0(x) \, d\nu(x)
    \, d\mu(y) \\
    &= \int \bigg(\sum_{j=1}^M v_j \Big( \frac{q_j(x)}{q_0(x)}-1\Big)\bigg)^2 q_0(x) \, d\nu(x)
    \\
    &=\int \bigg( \sum_{j=1}^M v_j \Big(\frac{dQ_j}{dQ_0}-1\Big) \bigg)^2 \, dQ_0.
\end{align*}
\end{proof}

\section{Derivations for explicit expressions for the \texorpdfstring{$\Ralpha$}{R\_alpha} codivergence}
\label{sec.explicit}

In this section we derive closed-form expressions for the $\Ralpha$ codivergences in Table~\ref{tab:explicit_expressions}. We also obtain a closed-form formula for the case of Gamma distributions and discuss a first order approximation of it.

\medskip

\subsection{Multivariate normal distribution}

Suppose $P_j=\Nc(\theta_j, \sigma^2 I_d)$
for $j=0, 1, 2.$
Here $\theta_j=(\theta_{j1},\dots,\theta_{jd})^\top$ are vectors in $\Rb^d$ and $I_d$ denotes the $d\times d$ identity matrix. Then,
\begin{align}
    \Ralpha(P_0 | P_1, P_2)
    = \exp\Big(\alpha^2\frac{\langle \theta_1 - \theta_0,
    \theta_2 - \theta_0 \rangle}{\sigma^2}\Big) -1.
    \label{eq.phialpha_for_normal}
\end{align}

\begin{proof}
The Lebesgue density of $P_j$ is 
\begin{align*}
    \frac 1{\sqrt{2\pi}}\exp\Big(-\frac {\|x-\theta_j\|^2}{2\sigma^2}\Big)=\frac 1{\sqrt{2\pi}} \exp\Big(-\frac {\|x\|^2}{2\sigma^2}\Big)\exp\Big(\frac 1{\sigma^2} \theta_j^\top x -\frac {\|\theta_j\|^2}{2\sigma^2}\Big),    
\end{align*}
with $\|\cdot\|$ the Euclidean norm. This is an exponential family $h(x)\exp(\langle \theta, T(x)\rangle -A(\theta))$ with $T(x) = \sigma^{-2}x$ and $A(\theta)=\|\theta\|^2/(2\sigma^2).$

Applying Proposition \ref{prop.exponential_family} and quadratic expansion $\|\theta_0+b\|^2=\|\theta_0\|^2+2\langle \theta_0, b\rangle+\|b\|^2$ to all four terms yields
\begin{align*}
    \Ralpha(P_0 | P_1, P_2)
    =
    &\exp\Big(\frac{\|\theta_0+\alpha(\theta_1+\theta_2-2\theta_0) \|^2}{2\sigma^2}
    - \frac{\|\theta_0+\alpha(\theta_1-\theta_0) \|^2}{2\sigma^2} \\
    &\quad\quad - \frac{\|\theta_0+\alpha(\theta_2-\theta_0) \|^2}{2\sigma^2}
    + \frac{\|\theta_0 \|^2}{2\sigma^2}\Big)-1 \\
    &= \exp\Big(\frac{\|\alpha(\theta_1+\theta_2-2\theta_0) \|^2- \|\alpha(\theta_1-\theta_0) \|^2 - \|\alpha(\theta_2-\theta_0) \|^2
    }{2\sigma^2}\Big)-1 \\
    &=\exp\Big(\alpha^2\frac{\langle \theta_1 - \theta_0,
    \theta_2 - \theta_0 \rangle}{\sigma^2}\Big) -1.
\end{align*}
\end{proof}

\subsection{Poisson distribution}

If $\Pois(\lambda)$ denotes the Poisson distribution with intensity $\lambda > 0,$ and $\lambda_0,\lambda_1,$ $\lambda_2 > 0$, then,
\begin{align*}
    \Ralphabig{\Pois(\lambda_0)}{
    \Pois(\lambda_1)}{
    \Pois(\lambda_2)}
    = \exp\Big(\lambda_0^{1 - 2\alpha}
    \big( \lambda_1^\alpha - \lambda_0^{\alpha} \big)
    \big( \lambda_2^\alpha - \lambda_0^{\alpha} \big) \Big) - 1.
\end{align*}
Suppose $P_j = \otimes_{\ell=1}^d \Pois(\lambda_{j\ell})$ for $j=0,\dots, M$ and $\lambda_{j\ell}>0$ for all $j,\ell.$
Then, as a consequence of Proposition~\ref{prop:phi_alpha_product},
\begin{align*}
    \Ralpha(P_0|P_1,P_2)
    = \exp\Big(\sum_{\ell=1}^d \lambda_{0\ell}^{1 - 2\alpha}
    \big( \lambda_{1\ell}^\alpha - \lambda_{0\ell}^{\alpha} \big)
    \big( \lambda_{2\ell}^\alpha - \lambda_{0\ell}^{\alpha} \big)
    \Big) - 1,
\end{align*}
with particular cases
\begin{align}
    \chi^2(P_0 | P_1, P_2)
    = \exp\Big( \sum_{\ell=1}^d \frac{(\lambda_{1\ell}-\lambda_{0\ell})(\lambda_{2\ell}-\lambda_{0\ell})}{\lambda_{0\ell}}\Big)-1,
\label{eq.chi2_matrix_for_Poisson}
\end{align}
and 
\begin{align}
    \rho(P_0 | P_1, P_2)
    = \exp\Big( \sum_{\ell=1}^d \big(\sqrt{\lambda_{1\ell}}-\sqrt{\lambda_{0\ell}}\big)\big(\sqrt{\lambda_{2\ell}}-\sqrt{\lambda_{0\ell}}\big)\Big) -1.
\label{eq.rho_matrix_for_Poisson}   
\end{align}

\begin{proof}
The density of the Poisson distribution with respect to the counting measure is
$$p_\lambda(x)
= e^{-\lambda} \frac{\lambda^x}{x!} 
= \frac{1}{x!}  e^{-\lambda + x \log(\lambda)} =h(x)\exp\big(\theta^\top T(x)-A(\theta)\big),$$
with $\theta = \log(\lambda)$, $T(x) = x$ and $A(\theta) = \exp(\theta)$.
Applying Proposition~\ref{prop.exponential_family} gives
\begin{align*}
    \Ralpha(P_{\theta_0}|P_{\theta_1},P_{\theta_2})
    &=\exp\Big(A\big(\theta_0+\alpha\big(\theta_1+\theta_2-2\theta_0\big)\big)
    - A\big(\theta_0+\alpha\big(\theta_1-\theta_0\big)\big) \\
    &\hspace{3em} - A\big(\theta_0+\alpha\big(\theta_2-\theta_0\big)\big)
    + A\big(\theta_0\big)\Big)-1. \\
    &= \exp\Big(\exp\big(\log(\lambda_0)
    + \alpha \big(\log(\lambda_1) + \log(\lambda_2) 
    - 2 \log(\lambda_0) \big)\big) \\
    &\hspace{3em}
    - \exp\big(\log(\lambda_0) 
    + \alpha \big(\log(\lambda_1) - \log(\lambda_0) \big)\big) \\
    &\hspace{3em}
    - \exp\big(\log(\lambda_0) 
    + \alpha \big(\log(\lambda_2) - \log(\lambda_0) \big)\big)
    + \lambda_0 \Big)-1. \\
    &= \exp\Big(
    \lambda_0^{1 - 2\alpha} \lambda_1^\alpha \lambda_2^\alpha
    - \lambda_0^{1 - \alpha} \lambda_1^\alpha
    - \lambda_0^{1 - \alpha} \lambda_2^\alpha
    + \lambda_0 \Big) - 1 \\
    &= \exp\Big(\lambda_0^{1 - 2\alpha}
    \big( \lambda_1^\alpha - \lambda_0^{\alpha} \big)
    \big( \lambda_2^\alpha - \lambda_0^{\alpha} \big) \Big) - 1.
\end{align*}
\end{proof}

\subsection{Bernoulli distribution}

If $\Ber(\theta)$ denotes the Poisson distribution with parameter $\theta \in (0,1),$ and $\theta_0,$ $\theta_1,$ $\theta_2 \in (0,1),$ then,
\begin{align*}
    &\Ralphabig{\Ber(\theta_0)}{\Ber(\theta_1)}{\Ber(\theta_2)} \\
    &= \frac{\theta_0^{1 - 2\alpha} \theta_1^\alpha \theta_2^\alpha
    + (1 - \theta_0)^{1 - 2 \alpha} (1 - \theta_1)^\alpha (1 - \theta_2)^\alpha}{
    \big( \theta_0^{1 - \alpha} \theta_1^\alpha
    + (1 - \theta_0)^{1 - \alpha} (1 - \theta_1)^\alpha \big)
    \big( \theta_0^{1 - \alpha} \theta_2^\alpha
    + (1 - \theta_0)^{1 - \alpha} (1 - \theta_2)^\alpha \big)
    } - 1,
\end{align*}
Suppose $P_j = \otimes_{\ell=1}^d \Ber(\theta_{j\ell})$ for $j=0, 1, 2$ and $\theta_{j\ell} \in (0,1)$ for all $j,\ell.$
Then, as a consequence of Proposition~\ref{prop:phi_alpha_product},
\begin{align*}
    \Ralpha(P_0|P_1,P_2)
    = \prod_{\ell=1}^d
    \frac{\theta_{0\ell}^{1 - 2\alpha} \theta_{1\ell}^\alpha \theta_{2\ell}^\alpha
    + (1 - \theta_{0\ell})^{1 - 2 \alpha} (1 - \theta_{1\ell})^\alpha (1 - \theta_{2\ell})^\alpha}{
    r(\theta_{0\ell}, \theta_{1\ell}) r(\theta_{0\ell}, \theta_{2\ell})
    } - 1,
\end{align*}
where $r(\theta_0,\theta_1) := \theta_0^{1 - \alpha} \theta_1^\alpha
+ (1 - \theta_0)^{1 - \alpha} (1 - \theta_1)^\alpha$. In particular,
\begin{align}
    \chi^2(P_0 | P_1, P_2) = 
    \prod_{\ell=1}^d
    \bigg(\frac{
    (\theta_{1\ell} - \theta_{0\ell})
    (\theta_{2\ell} - \theta_{0\ell}) }
    {\theta_{0\ell}(1-\theta_{0\ell})} + 1 \bigg) - 1,
\label{eq.chi2_matrix_for_Bernoulli}
\end{align}
and 
\begin{align}
    \rho(P_0 | P_1, P_2)
    = \prod_{\ell=1}^d \frac{\widetilde r(\theta_{1\ell},\theta_{2\ell})}
    {\widetilde r(\theta_{1\ell},\theta_{0\ell}) \widetilde r(\theta_{2\ell},\theta_{0\ell})} - 1,
\label{eq.rho_matrix_for_Bernoulli}   
\end{align}
with $\widetilde r(\theta,\theta')
:= \sqrt{\theta\theta'}+\sqrt{(1-\theta)(1-\theta')}.$

\begin{proof}
The Bernoulli distributions $\Ber(\theta), \theta \in (0,1)$ form an exponential family, dominated by the counting measure on $\{0, 1\}$ with density $P(\Ber(\theta) = k)
= \theta^k (1 - \theta)^{1 - k}
= \exp(k \log(\theta) + (1-k) \log(1 - \theta))
= \exp(k \beta - \log(1 + e^\beta) )$,
where $\beta = \log(\theta/(1-\theta))$ is the natural parameter and 
$A(\beta) = \log(1 + e^\beta)$.
Therefore, we can apply Proposition~\ref{prop.exponential_family} and obtain
\begin{align*}
    \Ralpha(P_{\beta_0}|P_{\beta_1},P_{\beta_2})
    = &\exp\Big(A\big(\beta_0+\alpha\big(\beta_1+\beta_2-2\beta_0\big)\big)
    - A\big(\beta_0+\alpha\big(\beta_1-\beta_0\big)\big) \\
    &\quad\quad - A\big(\beta_0+\alpha\big(\beta_2-\beta_0\big)\big)
    + A\big(\beta_0\big)\Big)-1.
\end{align*}
Note that
\begin{align*}
    \beta_0 + \alpha\big(\beta_1 - \beta_0\big)
    &= \log \left( \frac{\theta_0}{1 - \theta_0} \right)
    + \alpha \bigg( \log \left( \frac{\theta_1}{1 - \theta_1} \right)
    - \log \left( \frac{\theta_0}{1 - \theta_0} \right) \bigg) \\
    &= \log \left( \frac{\theta_0^{1 - \alpha} \theta_1^\alpha}
    {(1 - \theta_0)^{1 - \alpha} (1 - \theta_1)^\alpha} \right),
\end{align*}
so that
\begin{align*}
    A\big(\beta_0+\alpha\big(\beta_1-\beta_0\big)\big)
    &= \log \left(1 + \frac{\theta_0^{1 - \alpha} \theta_1^\alpha}
    {(1 - \theta_0)^{1 - \alpha} (1 - \theta_1)^\alpha} \right) \\
    &= \log \left(\frac{\theta_0^{1 - \alpha} \theta_1^\alpha
    + (1 - \theta_0)^{1 - \alpha} (1 - \theta_1)^\alpha}
    {(1 - \theta_0)^{1 - \alpha} (1 - \theta_1)^\alpha} \right).
\end{align*}
Similarly,
\begin{align*}
    &\beta_0 + \alpha\big(\beta_1 + \beta_2 - 2\beta_0\big)\\
    &= \log \left( \frac{\theta_0}{1 - \theta_0} \right)
    + \alpha \bigg( \log \left( \frac{\theta_1}{1 - \theta_1} \right)
    + \log \left( \frac{\theta_2}{1 - \theta_2} \right)
    - 2 \log \left( \frac{\theta_0}{1 - \theta_0} \right) \bigg) \\
    &= \log \left( \frac{\theta_0^{1 - 2\alpha} \theta_1^\alpha \theta_2^\alpha}
    {(1 - \theta_0)^{1 - 2 \alpha} (1 - \theta_1)^\alpha (1 - \theta_2)^\alpha} \right),
\end{align*}
so that
\begin{align*}
    A\big( \beta_0 + \alpha\big(\beta_1 + \beta_2 - 2\beta_0\big) \big)
    = \log \left( \frac{\theta_0^{1 - 2\alpha} \theta_1^\alpha \theta_2^\alpha
    + (1 - \theta_0)^{1 - 2 \alpha} (1 - \theta_1)^\alpha (1 - \theta_2)^\alpha}
    {(1 - \theta_0)^{1 - 2 \alpha} (1 - \theta_1)^\alpha (1 - \theta_2)^\alpha} \right).
\end{align*}
Combining all these results together yields
\begin{align*}
    \Ralpha(P_{\beta_0}&|P_{\beta_1},P_{\beta_2})
    = \exp\Bigg(
    \log \left( \frac{\theta_0^{1 - 2\alpha} \theta_1^\alpha \theta_2^\alpha
    + (1 - \theta_0)^{1 - 2 \alpha} (1 - \theta_1)^\alpha (1 - \theta_2)^\alpha}
    {(1 - \theta_0)^{1 - 2 \alpha} (1 - \theta_1)^\alpha (1 - \theta_2)^\alpha} \right) \\
    &\quad - \log \left(\frac{\theta_0^{1 - \alpha} \theta_1^\alpha
    + (1 - \theta_0)^{1 - \alpha} (1 - \theta_1)^\alpha}
    {(1 - \theta_0)^{1 - \alpha} (1 - \theta_1)^\alpha} \right) \\
    &\quad - \log \left(\frac{\theta_0^{1 - \alpha} \theta_2^\alpha
    + (1 - \theta_0)^{1 - \alpha} (1 - \theta_2)^\alpha}
    {(1 - \theta_0)^{1 - \alpha} (1 - \theta_2)^\alpha} \right)
    + \log(1 - \theta_0) \Bigg)  - 1\\
    &= \frac{\theta_0^{1 - 2\alpha} \theta_1^\alpha \theta_2^\alpha
    + (1 - \theta_0)^{1 - 2 \alpha} (1 - \theta_1)^\alpha (1 - \theta_2)^\alpha}{
    \big( \theta_0^{1 - \alpha} \theta_1^\alpha
    + (1 - \theta_0)^{1 - \alpha} (1 - \theta_1)^\alpha \big)
    \big( \theta_0^{1 - \alpha} \theta_2^\alpha
    + (1 - \theta_0)^{1 - \alpha} (1 - \theta_2)^\alpha \big)
    } - 1,
\end{align*}
finishing the proof.

\end{proof}

\subsection{Gamma distribution}
\label{sec:computation_Gamma_distr}

Let $P_\theta=\Gamma(\alpha, \beta)$ with $\theta=(\alpha-1,-\beta)$ denote the Gamma distribution with shape $\alpha > 0$ and inverse scale $\beta > 0$. If $\alpha_0, \alpha_1, \alpha_2, \beta_0, \beta_1, \beta_2, \alpha_0 + \alpha (\alpha_1 + \alpha_2 - 2 \alpha_0), \alpha_0 + \alpha (\alpha_1 - \alpha_0),\alpha_0 + \alpha (\alpha_2 - \alpha_0),\beta_0 + \alpha (\beta_1 + \beta_2 - 2 \beta_0),\beta_0 + \alpha (\beta_1 - \beta_0) > 0,$ and $\beta_0 + \alpha (\beta_2 - \beta_0)$  are all positive, then we have
\begin{align*}
    &\Ralpha(P_{\theta_0}|P_{\theta_1},P_{\theta_2})\\
    &= \frac{\Gamma(\alpha_0) \Gamma \big(\alpha_0
    + \alpha (\alpha_1 + \alpha_2 - 2 \alpha_0) \big)
    }{
    \Gamma \big(\alpha_0
    + \alpha (\alpha_1 - \alpha_0) \big)
    \Gamma \big(\alpha_0
    + \alpha (\alpha_2 - \alpha_0) \big)} \\
    &\quad \ \ \times \frac{
    \big(\beta_0
    + \alpha (\beta_1 - \beta_0) \big)^{\alpha_0
    + \alpha (\alpha_1 - \alpha_0)}
    \big(\beta_0
    + \alpha (\beta_2 - \beta_0) \big)^{\alpha_0
    + \alpha ( \alpha_2 - \alpha_0)}
    }{
    \beta_0^{\alpha_0}
    \big(\beta_0
    + \alpha (\beta_1 + \beta_2 - 2 \beta_0) \big)^{\alpha_0
    + \alpha (\alpha_1 + \alpha_2 - 2 \alpha_0)}
    }
    - 1,
\end{align*} 
otherwise $\Ralpha(P_{\theta_0}|P_{\theta_1},P_{\theta_2}) = +\infty$.
This can be checked by writing the explicit expression of the integrals that appear in the definition of $\Ralpha$.
%
%
\begin{proof}
The Gamma distributions $\Gamma(\alpha, \beta), \alpha > 0, \beta > 0$ form an exponential family, dominated by the Lebesgue measure with density
\begin{align*}
    \beta^{\alpha} x^{\alpha - 1} \exp(-\beta x)
    \Gamma (\alpha)^{-1}
    = \exp \big( (\alpha - 1) \log(x)
    - \beta x
    + \alpha \log(\beta)
    - \log(\Gamma (\alpha))
    \big),
\end{align*}
natural parameter
\begin{align*}
    \theta = (\theta_{a}, \theta_{b})
    = (\alpha - 1, - \beta),
\end{align*}
and 
\begin{align*}
    A(\theta) = \log\big(\Gamma (\theta_{a} + 1)\big)
    - (\theta_{a} + 1) \log(- \theta_{b}).
\end{align*}
Therefore, we can apply Proposition~\ref{prop.exponential_family} 
with $\Theta = (-1, +\infty) \times (- \Rb)$.
Combining the assumed constraints on the parameters and the linearity of the mapping $(\alpha, \beta) \mapsto \theta$ ensures that 
$\theta_0+\alpha\big(\theta_1-\theta_0\big),$
$\theta_0+\alpha\big(\theta_1-\theta_0\big),$
$\theta_0+\alpha\big(\theta_1+\theta_2-2\theta_0\big) \in \Theta$.
Therefore, we obtain
\begin{align*}
    \Ralpha(P_{\theta_0}|P_{\theta_1},P_{\theta_2})
    &=\exp\Big(A\big(\theta_0+\alpha\big(\theta_1+\theta_2-2\theta_0\big)\big)
    - A\big(\theta_0+\alpha\big(\theta_1-\theta_0\big)\big) \\
    &\hspace{3em} - A\big(\theta_0+\alpha\big(\theta_2-\theta_0\big)\big)
    + A\big(\theta_0\big)\Big)-1,
\end{align*}
where
\begin{align*}
    A\big(\theta_0+\alpha\big(\theta_1-\theta_0\big)\big)
    &= \log\Big(\Gamma \Big(\theta_{0, a}
    + \alpha \big(\theta_{1,a} - \theta_{0,a} \big) + 1
    \Big) \Big) \\
    &- \Big(\theta_{0, a}
    + \alpha \big(\theta_{1,a} - \theta_{0,a} \big) + 1
    \Big) \log \Big(\beta_0
    + \alpha \big(\beta_1 - \beta_0 \big) \Big) \\
    %
    &= \log\Big(\Gamma \Big(\alpha_0
    + \alpha \big( \alpha_1 - \alpha_0 \big)
    \Big) \Big) \\
    &- \Big(\alpha_0
    + \alpha \big( \alpha_1 - \alpha_0 \big)
    \Big) \log \Big(\beta_0
    + \alpha \big(\beta_1 - \beta_0 \big) \Big)
\end{align*}
and
\begin{align*}
    A\big(\theta_0+\alpha\big(\theta_1+\theta_2-2\theta_0\big)
    &= \log\Big(\Gamma \Big(\alpha_0
    + \alpha \big( \alpha_1 + \alpha_2 - 2 \alpha_0 \big)
    \Big) \Big) \\
    &- \Big(\alpha_0
    + \alpha \big( \alpha_1 + \alpha_2 - 2 \alpha_0 \big)
    \Big) \log \Big(\beta_0
    + \alpha \big(\beta_1 + \beta_2 - 2 \beta_0 \big) \Big).
\end{align*}
Combining all these results, we obtain
\begin{align*}
    &\Ralpha(P_{\theta_0}|P_{\theta_1},P_{\theta_2}) \\
    &= \frac{\Gamma(\alpha_0) \Gamma \Big(\alpha_0
    + \alpha \big( \alpha_1 + \alpha_2 - 2 \alpha_0 \big)
    \Big)
    }{
    \Gamma \Big(\alpha_0
    + \alpha \big( \alpha_1 - \alpha_0 \big)
    \Big)
    \Gamma \Big(\alpha_0
    + \alpha \big( \alpha_2 - \alpha_0 \big)
    \Big)} \\
    &\quad \ \ \times \frac{
    \big(\beta_0
    + \alpha (\beta_1 - \beta_0) \big)^{\alpha_0
    + \alpha (\alpha_1 - \alpha_0)}
    \big(\beta_0
    + \alpha (\beta_2 - \beta_0) \big)^{\alpha_0
    + \alpha ( \alpha_2 - \alpha_0)}
    }{
    \beta_0^{\alpha_0}
    \big(\beta_0
    + \alpha (\beta_1 + \beta_2 - 2 \beta_0) \big)^{\alpha_0
    + \alpha (\alpha_1 + \alpha_2 - 2 \alpha_0)}
    }
    - 1.
\end{align*}
\end{proof}

A formula for the product of exponential distributions can be obtained as a special case by setting $\alpha_{j\ell}=1$ for all $j,\ell$
and applying Proposition~\ref{prop:phi_alpha_product}.
For the families of distributions discussed above, the formulas for the correlation-type $\Ralpha$ codivergences encode an orthogonality relation on the parameter vectors.
This is less visible in the expressions for the Gamma distribution but can be made more explicit using the first order approximation that we state next.
It shows that even for the Gamma distribution these matrix entries can be written in leading order as a term involving a weighted inner product of $\beta_1-\beta_0$ and $\beta_2-\beta_0,$
where $\beta_r$ denotes the vector $(\beta_{r\ell})_{1 \leq \ell \leq d}.$

\begin{lemma}\label{lem.Gamma_expansion}
    Suppose $P_j = \otimes_{\ell=1}^d 
    \Gamma(\alpha_{\ell},\beta_{j\ell})$ for every $j=1,2,3$ and for some $\alpha_{\ell},\beta_{j\ell} > 0$.
    Let $A:=\sum_{\ell=1}^d \alpha_\ell$ and $\Delta:=\max_{j=1,2}\max_{\ell=1,\ldots,d} |\beta_{j\ell}-\beta_{0\ell}|/\beta_{0\ell}.$
    Denote by $\Sigma$ the $d\times d$ diagonal matrix with entries $\beta_{0\ell}^2/\alpha_{\ell}.$ Then, 
    \begin{align*}
        \Ralpha(P_0|P_1,P_2)
        = \exp \Big( - \alpha^2
        (\beta_1-\beta_0)^\top \Sigma^{-1}(\beta_2-\beta_0) + o(A \Delta^2)\Big)-1.
    \end{align*}
\end{lemma}

\begin{proof}
Using that $\alpha_\ell$ does not depend on $j,$
the expression
simplifies and a second order Taylor expansion of the logarithm (the sum of the first order terms vanishes) yields
\begin{align*}
    &\Ralpha(P_{\theta_0}|P_{\theta_1},P_{\theta_2}) \\
    &= \prod_{\ell=1}^d
    \frac {
    (\beta_{1\ell} + \beta_{0\ell})^{\alpha_\ell}
    (\beta_{2\ell} + \beta_{0\ell})^{\alpha_\ell}
    }
    {(2 \beta_{0\ell})^{\alpha_{\ell}}
    (\beta_{1\ell} + \beta_{2\ell})^{\alpha_\ell}} \\
    &= \prod_{\ell=1}^d
    \exp \Bigg( \alpha_\ell \bigg(
    \log \bigg(1 + \alpha \frac{\beta_{1\ell} - \beta_{0\ell}}{ \beta_{0\ell}}
    \bigg)
    + \log \bigg(1 + \alpha \frac{\beta_{2\ell} - \beta_{0\ell}}{\beta_{0\ell}}
    \bigg) \\
    &\hspace{16.5em}
    - \log \bigg(1 + \alpha \frac{\beta_{1\ell} - \beta_{0\ell}
    + \beta_{2\ell} - \beta_{0\ell}}{\beta_{0\ell}}
    \bigg) \bigg) \Bigg) \\
    &= \exp \Bigg( \sum_{\ell=1}^d \alpha_\ell \alpha^2 \Bigg(
    - \frac{ (\beta_{1\ell}-\beta_{0\ell})^2}{2\beta_{0\ell}^2}
    - \frac{(\beta_{2\ell}-\beta_{0\ell})^2}{2\beta_{0\ell}^2} 
    + \frac{(\beta_{1\ell}-\beta_{0\ell} + \beta_{2\ell}-\beta_{0\ell})^2}{2\beta_{0\ell}^2}
     + o(\Delta^2)\bigg) \Bigg) \\
    &= \exp \Bigg( \alpha^2 \sum_{\ell=1}^d
    \frac{\alpha_\ell (\beta_{1\ell}-\beta_{0\ell}) (\beta_{2\ell}-\beta_{0\ell})}{\beta_{0\ell}^2}
    + o\big(A \Delta^2\big) \Bigg).
\end{align*}
\end{proof}

\section{Facts about ranks}
\label{sec:useful_lemmas}

\begin{defi}
    Let $X_1, \dots, X_n$ be $n$ random variables defined on the same probability space $(\Omega, \Ac, P)$. We define the \emph{rank} of $\{X_1, \dots, X_n\}$, denoted by $\Rank(X_1, \dots, X_n)$ as the dimension of the vector space $\Vect(X_1, \dots, X_n)$ generated by linear combinations of $\{X_1, \dots, X_n\}$, where the equality is to be understood $P$-almost surely. Moreover, we say that $(X_1, \dots, X_n)$ are \emph{linearly independent $P$-almost surely} if for any vector $(a_1, \dots, a_n),$
    \begin{align*}
        P \bigg( \sum_{i=1}^n a_i X_i = 0 \bigg) = 1
        \quad \text{implies} \ \  a_1 = \cdots = a_n = 0.
    \end{align*}
\label{def:rank_random_variables}
\end{defi}

\begin{lemma}
    Let $X_1, \dots, X_n$ be $n$ random variables defined on the same probability space $(\Omega, \Ac, P)$.
    $\Rank(X_1, \dots, X_n)$ is the largest integer such that there exists $i_1, \dots, i_r \in \{1, \dots, M\}$ with $(X_{i_1}, \dots, X_{i_r})$ linearly independent random variables $P$-almost surely.
\label{lemma:rank_max_lincombin}
\end{lemma}

\begin{proof}
    Let $r$ be the largest integer such that there exists $i_1, \dots, i_r \in \{1, \dots, M\}$ with $(X_{i_1}, \dots, X_{i_r})$ linearly independent random variables $P$-almost surely.
    Then the space generated by $X_1, \dots, X_n$ is at least of dimension $r$, and therefore $\Rank(X_1, \dots, X_n) \geq r$.
    If $\Rank(X_1, \dots, X_n) > r$, then there exists $(r+1)$ linear combinations of the random variables that are linearly independent, contradicting the definition of $r$. Therefore $\Rank(X_1, \dots, X_n) \leq r$, completing the proof.
\end{proof}

\begin{lemma}
    Let $\Z = (Z_1,\ldots,Z_M)^\top$ be a $M$-dimensional random vector with mean zero and finite second moments.
    Then $\Rank(\Cov_P(\Z)) = \Rank(Z_1, \dots, Z_M)$,
    where the covariance matrix is computed with respect to the distribution $P$ and the rank of a set of random variables is to be understood in the sense of Definition~\ref{def:rank_random_variables}.
\label{lemma:rank_CovZ}
\end{lemma}

\begin{proof}
    Let $\lambda_1 \geq \lambda_2 \geq \dots \geq \lambda_M$ be the eigenvalues of $\Cov_P(\Z)$, sorted in decreasing order, and let $\e_1, \dots, \e_M$ be a corresponding orthonormal basis of eigenvectors.
    Let $r$ be the rank of $\Cov_P(\Z)$.
    We have $\lambda_{r+1} = \lambda_{r+2} = \cdots = \lambda_M = 0$ and $\lambda_r > 0$.
    Let us define $Y_i = \e_i^\top \Z$ for $i=1, \dots, M$.
    By usual results on principal components, e.g. \cite[Result 8.1]{johnson2007applied}, $\Var[Y_i] = \lambda_i$ and $\Cov(Y_i, Y_j) = \lambda_i 1_{\{i=j\}}$.
    Therefore,
    \begin{align*}
        \Rank(Z_1, \dots, Z_M)
        &= \dim(\Vect(Z_1, \dots, Z_M))
        = \dim(\Vect(Y_1, \dots, Y_M)) \\
        & = \dim(\Vect(Y_1, \dots, Y_r)) = r,
    \end{align*}
    where the first equality is the definition of the rank, the second equality is a consequence of the fact that $(\e_1, \dots, \e_M)$ is a basis of $\Rb^M$, the third equality results from the fact that $\Var[Y_i] = 0$ and $E[Y_i] = 0$ for any $i > r$ and the last equality is a consequence of the orthogonality of the $(Y_1, \dots, Y_r)$ as elements of the Hilbert space $L_2(\Omega, \Ac, P)$.
    The proof is completed since by definition $r = \Rank(\Cov_P(\Z))$.
\end{proof}

\begin{lemma}[see for example Exercise 3.3.11 in \cite{rao2000linear}]
    A symmetric and positive semi-definite $M\times M$ matrix $A$ is of rank $r$ if and only if $A$ has an invertible principal submatrix of size $r$, and all principal submatrices of size $r+1$ of $A$ are singular.
\label{lemma:rank_principal_minors}
\end{lemma}

\section{Conclusion}

We introduced the concept of codivergence as a notion of ``angle'' between three probability distributions. Divergence matrices can be viewed as an analogue of the Gram matrix for a finite sequence of probability distributions that are compared relative to one distribution.

\medskip

Locally around the reference probability measure $P_0$, codivergences are bilinear forms up to remainder terms. Two classes of codivergences emerge that resemble the structure of the covariance and the correlation.

\medskip

Natural follow-up questions relate to the spectral behavior of a divergence matrix and the link between properties of the divergence matrix and properties of the underlying probability measures.



\backmatter





\bmhead{Acknowledgments}
We are grateful to the Associate Editor and two referees for valuable comments, suggesting Proposition \ref{prop.exponential_family}, and an idea that led us consider the two general classes of covariance-type and correlation-type codivergences.



\section*{Declarations}

\textbf{Funding}: The research has been supported by the NWO/STAR grant 613.009.034b and the NWO Vidi grant VI.Vidi.192.021.

\smallskip

\noindent
\textbf{Competing interests}: On behalf of all authors, the corresponding author states that there is no conflict of interest.

\smallskip

\noindent
\textbf{Data availability}: Data sharing not applicable to this article as no datasets were generated or analysed during the current study.

\smallskip

\noindent
\textbf{Author's contributions}: Both authors contributed equally to this work.









\begin{appendices}






\section{Proofs}

\subsection{Proof of Proposition~\ref{prop:codiv}}
\label{proof:prop:codiv}

\begin{proof}
As mentioned already, the first and second part of Definition~\ref{def:codiv} are satisfied.
To check the third part of Definition~\ref{def:codiv} for $\phi(P_0|P_1,P_2),$ let $\mu, \wt \mu \in \Mc_{P_0}$. Then
\begin{align*}
    \frac{d(P_0 + t \mu)}{dP_0}
    = 1 + t \frac{d\mu}{dP_0} = 1 + t h,
\end{align*}
is square-integrable with respect to $P_0$ for any real number $t$. Using $\phi(1)=1$, Taylor expansion $\phi(1+y)=1+y \phi'(1) +\tfrac{1}{2}y^2\phi''(1) + o(y^2),$ that $\int h dP_0=0$ and that $h$ is bounded $P_0$-a.e. by the definition of $\Mc_{P_0},$ we obtain that, for all sufficiently small $t,$
\begin{align*}
    \int \phi\Big( \frac{d(P_0 + t \mu)}{dP_0}\Big) \, dP_0
    = \int \phi(1+th) \, dP_0
    = 1 + \frac{t^2}{2}\phi''(1)\int h^2 \, dP_0+o(t^2).
\end{align*}
Similarly $\int \phi\big( \frac{d(P_0 + s \wt \mu)}{dP_0}\big) \, dP_0= 1+\tfrac{1}{2}s^2\phi''(1)\int g^2 \, dP_0+o(s^2)$ and 
\begin{align*}
    &\int \phi\Big( \frac{d(P_0 + t \mu)}{dP_0}\Big)\phi\Big( \frac{d(P_0 + s \wt \mu)}{dP_0}\Big) \, dP_0\\
    &= \int \phi(1+th)\phi(1+sg) \, dP_0 \\
    &= 1+\frac {\phi''(1)}2 \int (t^2 h^2+s^2g^2) \, dP_0 + st \phi'(1)^2\int gh \, dP_0 +o(t^2+s^2) \\
    &= \Big(1+\frac{t^2}{2}\phi''(1)\int h^2 \, dP_0\Big)\Big(1+\frac{s^2}{2}\phi''(1)\int g^2 \, dP_0\Big)+ st \phi'(1)^2 \int gh \, dP_0 +o(t^2+s^2).
\end{align*}
Taylor expansion yields $1/(1-y) = 1 + O(y)$ for all $|y|\leq 1/2$ and thus for $(s,t)\to (0,0),$
\begin{align*}
    &\Rphi(P_0 | P_0 + t \mu, P_0 + s \wt \mu) \\
    &= \frac{st \phi'(1)^2 \int gh \, dP_0 + o(t^2+s^2)}{(1+\frac{t^2}{2}\phi''(1)\int h^2 \, dP_0)(1+\frac{s^2}{2}\phi''(1)\int g^2 \, dP_0) +o(t^2+s^2)} \\
    &= st \phi'(1)^2 \int gh \, dP_0 + o(t^2+s^2).
\end{align*}
By following the same arguments and replacing the definition of $\Rphi(P_0 | P_0 + t \mu, P_0 + s \wt \mu)$ by $\Vphi(P_0 | P_0 + t \mu, P_0 + s \wt \mu) $ in the last step, we also obtain $\Vphi(P_0 | P_0 + t \mu, P_0 + s \wt \mu)=st \phi'(1)^2 \int gh \, dP_0 + o(t^2+s^2).$
\end{proof}

\subsection{Proof of Proposition~\ref{prop:phi_alpha_product}}
\label{proof:prop:phi_alpha_product}

\begin{proof}
By Fubini's theorem,
\begin{align*}
    &\RalphaBigg{
    \bigotimes_{\ell=1}^d P_{0\ell}}{
    \bigotimes_{\ell=1}^d P_{1\ell}}{
    \bigotimes_{\ell=1}^d P_{2\ell}}+1  \\
    &= \dfrac{\displaystyle \int \left(
    \frac{d\bigotimes_{\ell=1}^d P_{1\ell}} {d\bigotimes_{\ell=1}^d P_{0\ell}} \right)^\alpha
    \left(
    \frac{d\bigotimes_{\ell=1}^d P_{2\ell}} {d\bigotimes_{\ell=1}^d P_{0\ell}} \right)^\alpha
    d \left(\bigotimes_{\ell=1}^d P_{0\ell} \right)
    }{\displaystyle \int \left(
    \frac{d\bigotimes_{\ell=1}^d P_{1\ell}} {d\bigotimes_{\ell=1}^d P_{0\ell}} \right)^\alpha
    d \left(\bigotimes_{\ell=1}^d P_{0\ell} \right)
    \int \left(
    \frac{d\bigotimes_{\ell=1}^d P_{2\ell}} {d\bigotimes_{\ell=1}^d P_{0\ell}} \right)^\alpha
    d \left(\bigotimes_{\ell=1}^d P_{0\ell} \right)
    } \\
    &= \prod_{\ell=1}^d \dfrac{\displaystyle \int \left(
    \frac{dP_{1\ell}} {dP_{0\ell}} \right)^\alpha
    \left(\frac{dP_{2\ell}} {dP_{0\ell}} \right)^\alpha
    dP_{0\ell}
    }{\displaystyle \int \left(
    \frac{d P_{1\ell}} {dP_{0\ell}} \right)^\alpha
    dP_{0\ell}
    \int \left(
    \frac{dP_{2\ell}} {dP_{0\ell}} \right)^\alpha
    dP_{0\ell}
    } \\
    &= \prod_{\ell=1}^d \Ralpha(
    P_{0\ell} |
    P_{1\ell},
    P_{2\ell})+1.
\end{align*}
\end{proof}

\subsection{Proof of Proposition~\ref{prop.exponential_family}}
\label{proof:prop.exponential_family}

\begin{proof}
Write $P_{i}:=P_{\theta_i}$ and $p_i$ for the corresponding $\nu$-densities. By assumption, we have $\overline \theta_\alpha:=\alpha (\theta_1+\theta_2)+(1-2\alpha)\theta_0 \in \Theta$ and
\begin{align*}
    \int \Big(\frac{dP_1}{dP_0}\Big)^\alpha \Big(\frac{dP_2}{dP_0}\Big)^\alpha dP_0
    &= \int \big(p_1(x)p_2(x)\big)^\alpha p_0(x)^{1-2\alpha} \, d\nu(x) \\
    &= \int h(x) 
    \exp\big( \overline \theta_\alpha^\top T(x)\big) \, d\nu(x) \\
    &\quad \cdot \exp\big(-\alpha A(\theta_1)-\alpha A(\theta_2)+(1-2\alpha)A(\theta_0)\big) \\
    &= \exp\big(A(\overline \theta_\alpha)-\alpha A(\theta_1)-\alpha A(\theta_2)+(1-2\alpha)A(\theta_0)\big).
\end{align*}
Setting $P_2=P_0$ (or equivalently, $\theta_2=\theta_0$) in the previous identity gives
\begin{align*}
    \int \Big(\frac{dP_1}{dP_0}\Big)^\alpha dP_0
    &= \exp\big(A\big(\alpha \theta_1+(1-\alpha)\theta_0\big)-\alpha A(\theta_1)+(1-\alpha)A(\theta_0)\big).
\end{align*}
Interchanging the role of $\theta_2$ and $\theta_1$ provides moreover a closed-form expression for $\int \big(\frac{dP_2}{dP_0} \big)^\alpha dP_0.$ Combining everything yields 
\begin{align*}
    &\Ralpha(P_0|P_1,P_2)
    = \dfrac{\int \big(\frac{dP_1}{dP_0}\big)^\alpha \big(\frac{dP_2}{dP_0}\big)^\alpha dP_0}{\int \big(\frac{dP_1}{dP_0}\big)^\alpha dP_0 \, \int \big(\frac{dP_2}{dP_0}\big)^\alpha dP_0} -1 \\
    &= \exp\Big(A\big(\overline \theta_\alpha\big)
    - A\big(\theta_0+\alpha\big(\theta_1-\theta_0\big)\big)
    - A\big(\theta_0+\alpha\big(\theta_2-\theta_0\big)\big)
    + A\big(\theta_0\big)\Big)-1.
\end{align*}
\end{proof}

\end{appendices}


\bibliography{biblio}


\end{document}